\documentclass[psamsfonts]{amsart}
\usepackage{amssymb,amsfonts,lacromay}
\usepackage[usenames,dvipsnames]{color}
\usepackage[all,arc]{xy}
\usepackage{enumerate}
\usepackage{mathrsfs}
\usepackage{stmaryrd}

\theoremstyle{thm}
\newtheorem{thm}{Theorem}[section]
\newtheorem{cor}[thm]{Corollary}
\newtheorem{prop}[thm]{Proposition}
\newtheorem{lem}[thm]{Lemma}

\newtheorem{quest}[thm]{Question}

\theoremstyle{definition}
\newtheorem{defn}[thm]{Definition}

\newtheorem{con}[thm]{Construction}

\newtheorem{exmp}[thm]{Example}

\newtheorem{notns}[thm]{Notations}

\newtheorem{asses}[thm]{Assumptions}
\newtheorem{rem}[thm]{Remark}

\makeatletter
\let\c@equation\c@thm
\makeatother
\numberwithin{equation}{section}

\makeatletter
\ifx\SK@label\undefined\let\SK@label\label\fi
 \let\your@thm\@thm
 \def\@thm#1#2#3{\gdef\currthmtype{#3}\your@thm{#1}{#2}{#3}}
 \def\mylabel#1{{\let\your@currentlabel\@currentlabel\def\@currentlabel
  {\currthmtype~\your@currentlabel}
 \SK@label{#1@}}\label{#1}}
 \def\myref#1{\ref{#1@}}
\makeatother

\title{Enriched model categories and presheaf categories}

\author{Bertrand Guillou}
\email{bertguillou@uky.edu}
\address{Department of Mathematics,
University of Kentucky, 
Lexington, KY 40506}

\author{J.P. May}
\email{may@math.uchicago.edu}
\address{Department of Mathematics,
The University of Chicago,
Chicago, IL 60637}


\begin{document}

\begin{abstract}
We collect in one place a variety of known and folklore
results in enriched model category theory and add a few
new twists.  The central theme is a general procedure for 
constructing a Quillen adjunction, often a Quillen equivalence, 
between a given $\sV$-model category and a category of enriched
presheaves in $\sV$, where $\sV$ is any good enriching category. 
For example, we rederive the result of
Schwede and Shipley that reasonable stable model categories
are Quillen equivalent to presheaf categories of spectra
(alias categories of module spectra) under more general
hypotheses.  The technical improvements 
and modifications of general model categorical results given 
here are applied to equivariant contexts in the
sequels \cite{GM2, GMequiv}, where we indicate various
directions of application.
\end{abstract}

\maketitle

\tableofcontents

\section*{Introduction}

The categories, $\sM$ say, that occur in nature have both hom 
sets $\sM(X,Y)$ and enriched hom objects $\ul{\sM}(X,Y)$ in 
some related category, $\sV$ say.  Technically $\sM$ is enriched 
over $\sV$.  In topology, the 
enrichment is often given simply as a topology on the set of
maps between a pair of objects, and its use is second nature.
In algebra, enrichment in abelian groups is similarly familiar
in the context of additive and Abelian categories.
In homological algebra, this becomes enrichment in chain complexes,
and the enriched categories go under the name of 
DG-categories.

Quillen's model category theory encodes homotopical algebra in 
general categories.  In and of itself, it concerns just the 
underlying category, but the relationship with the enrichment
is of fundamental importance in nearly all of the applications.

The literature of model category theory largely focuses on enrichment
in the category of simplicial sets and related categories with a 
simplicial flavor. Although there are significant technical advantages 
to working simplicially, as we shall see, the main reason for this is 
nevertheless probably more 
historical than mathematical.  Simplicial enrichment often occurs naturally, 
but it is also often arranged artificially, replacing a different naturally 
occurring enrichment with a simplicial one.  This is very natural if one's 
focus is on, for example, categories of enriched categories and all-embracing
generality.  It is not very
natural if one's focus is on analysis of, or calculations in, a particular
model category that comes with its own intrinsic enrichment.

The focus on simplicial enrichment gives a simplicial
flavor to the literature that perhaps impedes the wider dissemination
of model theoretic techniques.  For example, it can hardly be expected
that those in representation theory and other areas that deal naturally
with DG-categories will read much of the simplicially oriented model category 
literature, even though it is directly relevant to their work. 

Even in topology, 
it usually serves no mathematical purpose to enrich simplicially in situations 
in equivariant, parametrized, and classical homotopy theory that arise in nature 
with topological enrichments.  We recall a nice joke of John Baez when given a 
simplicial answer to a topological question.

\vspace{1mm}

{\footnotesize

``The folklore is fine as long as we really {\em can} interchange topological
spaces and simplicial sets. 

Otherwise it's a bit like this:

\vspace{1mm}

`It doesn't matter if you take a cheese sandwich or a ham sandwich;
  they're equally good.'

`Okay, I'll take a ham sandwich.'

`No!  Take a cheese sandwich - they're equally good.'

\vspace{1mm}

One becomes suspicious\dots''

}

\vspace{1mm}

Technically, however, there is very good reason for focusing on simplicial
enrichment: simplicity. The model category of simplicial sets enjoys special properties
that allow general statements about simplicially enriched model categories, unencumbered by
annoying added and hard to remember hypotheses that are necessary when enriching in a category 
$\sV$ that does not satisfy these properties.  Lurie \cite[A.3.2.16]{Lurie} defined the notion of an 
``excellent'' enriching category and restricted to those in his treatment \cite[A.3.3]{Lurie}
of diagram categories. In effect, that definition encodes the relevant special properties of simplicial 
sets.  None of the topological and few of the algebraic examples of interest to us are 
excellent.  These properties preclude other desirable properties.  For example, in algebra and 
topology it is often helpful to work with enriching categories in which all objects are fibrant,
whereas every object is cofibrant in an excellent enriching category.

While we also have explicit questions in mind, one of our goals is to
summarize and explain some of how model category theory works in general
in enriched contexts, adding a number of technical refinements that we need 
in the sequels \cite{GM2, GMequiv} and could not find in the literature. Many of our results appear in one form
or another in the standard category theory sources (especially Kelly \cite{Kelly}
and Borceux \cite{BorII}) and in the model theoretic work of Dugger, Hovey, Lurie, Schwede, 
and Shipley \cite{Dugger, DS, Hovey, Lurie, SS0, SS, SS2}. Although the latter papers largely 
focus on simplicial contexts, they contain the original versions and forerunners of many 
of our results. 

Cataloging the technical hypotheses needed to work with a general $\sV$ is tedious and 
makes for tedious reading. To get to more interesting things first, we follow a
referee's suggestion and work backwards.  We recall background material 
that gives the basic framework at the end. Thus we discuss enriched model 
categories, called $\sV$-model categories (see \myref{monoidal}), in general 
in \S\ref{enriched} and we discuss enriched diagram categories in  \S\ref{CatPre}.
The rest of \S\ref{diagram} gives relevant categorical addenda not used earlier. Thus
\S\ref{overV} and \S\ref{overM} describe ways of constructing maps from small $\sV$-categories 
into full $\sV$-subcategories of $\sV$ or, more generally, $\sM$, and \S\ref{mult} discusses
prospects for multiplicative elaborations of our results.  

Our main focus is the comparison between given
enriched categories and related categories of enriched presheaves.
{We are especially interested in examples where, in contrast to modules over a commutative
monoid in $\sV$, the model category $\sM$ requires more than one ``generator'', as is typical of
equivariant contexts \cite{GM2, GMequiv}.  We shall see in \cite{GMequiv} that the notion of ``equivariant 
contexts'' admits a considerably broader interpretation than just the study of group actions.}

We will discuss answers to the following questions in general terms in \S\ref{comparisons}.
They are natural variants on the theme of understanding the relationship
between model categories in general and model categories of enriched presheaves.
When $\sV$ is the category $sSet$ of simplicial sets, a version of the first 
question was addressed by Dwyer and Kan \cite{DK}.  Again when $\sV=sSet$, a 
question related to the second was addressed by Dugger \cite{Dug0, Dug}.  
When $\sV$ is the category $\SI\sS$ of symmetric spectra, the third question 
was addressed by Schwede and Shipley \cite{SS}.  
In the DG setting, an instance of the third question was addressed in \cite[\S7]{KthyDer}.

In all four questions, $\sD$ 
denotes a small $\sV$-category. The only model structure on presheaf categories
that concerns us in these questions is the projective level model structure induced from a 
given model structure on $\sV$: a map $f\colon X\rtarr Y$ of presheaves is a weak equivalence 
or fibration if and only if $f_d\colon X_d\rtarr Y_d$ is a weak equivalence or fibration for each
object $d$ of $\sD$; the cofibrations are
the maps that satisfy the left lifting property (LLP) with respect to the acyclic fibrations.
There is an evident dual notion of an injective model structure, but that will not concern us here.  
We call the projective level model structure the level model structure in this paper.

\begin{quest}\mylabel{quest1} Suppose that $\sM$ is a $\sV$-category
and $\de \colon \sD\rtarr \sM$ is a $\sV$-functor.  When can one use
$\de$ to define 
a $\sV$-model structure on $\sM$ such that $\sM$ is Quillen equivalent 
to the $\sV$-model category $\mathbf{Pre}(\sD,\sV)$ of enriched presheaves 
$\sD^{\text{op}}\rtarr \sV$?
\end{quest}

\begin{quest}\mylabel{quest2} Suppose that $\sM$ is a $\sV$-model category.  
When is $\sM$ Quillen equivalent to $\mathbf{Pre} (\sD,\sV)$, where $\sD$ is the full
$\sV$-subcategory of $\sM$ given by some well chosen set of objects $d\in \sM$?
\end{quest}

\begin{quest}\mylabel{quest3} Suppose that $\sM$ is a $\sV$-model category,
where $\sV$ is a {\em stable} model category. When is $\sM$ Quillen equivalent 
to $\mathbf{Pre} (\sD,\sV)$, where $\sD$ is the full $\sV$-subcategory of $\sM$ given 
by some well chosen set of objects $d\in \sM$?
\end{quest}

\begin{quest}\mylabel{quest4} More generally, we can ask Questions~\ref{quest2} 
and \ref{quest3}, but seeking a Quillen equivalence between 
$\sM$ and $\mathbf{Pre} (\sD,\sV)$ for some $\sV$-functor $\de\colon \sD\rtarr \sM$, 
not necessarily the inclusion of a full $\sV$-subcategory.
\end{quest}

Our answer to \myref{quest3} is  a variant of a theorem of Schwede and Shipley \cite{SS}.
It will play a central role in the sequel \cite{GM2}, where we give a convenient presheaf model for the 
category of $G$-spectra for any finite group $G$.

We are interested in \myref{quest4} since we shall see in \cite{GMequiv} that there are interesting $\sV$-model 
categories $\sM$ that are Quillen equivalent to  presheaf categories $\mathbf{Pre} (\sD,\sV)$, 
where $\sD$ is not a full subcategory of $\sM$ but, as far as we know, are not Quillen equivalent 
to a presheaf category $\mathbf{Pre} (\sD,\sV)$ for any full subcategory $\sD$ of $\sM$.  

We return to the general theory in  \S\ref{changeDM} and \S\ref{changeDMV}, 
where we give a variety of results that show 
how to change $\sD$, $\sM$, and $\sV$ without changing the Quillen equivalence 
class of the model categories we are interested in. Many of these results are technical 
variants or generalizations (or sometimes just helpful specializations) of results of 
Dugger, Hovey, Schwede, and Shipley \cite{Dugger, DS, Hovey, SS0, SS, SS2}. 
Some of these results are needed in \cite{GM2, GMequiv} and others are not, 
but we feel that a reasonably thorough compendium in one place may well be a service to others.
The results in this direction are
scattered in the literature, and they are important in applications of model category theory in a variety of contexts.  
The new notion of a tensored adjoint pair in \S\ref{tensored}
is implicit but not explicit in the literature and captures a commonly occurring phenomenon of enriched adjunction. The new notions of weakly unital $\sV$-categories and presheaves in
\S\ref{weakly} describe a phenomenon that appears categorically when the unit $\bf{I}$ of the symmetric monoidal model category $\sV$ is not cofibrant and appears topologically in 
connection with Atiyah duality, as we will explain in \cite{GM2}.

The basic idea is that $\sV$ is in practice a 
well understood model category, as are presheaf categories with values in
$\sV$. Modelling a general model category $\sM$ in terms of such a presheaf
category, with its elementary levelwise model structure, can be very useful 
in practice, as many papers in the literature make clear.  It is important
to the applications to understand exactly what is needed for such modelling
and how one can vary the model.   
We were led to our general questions by specific topological applications \cite{GM2, GMequiv},
but there are many disparate contexts where they are of interest.  

Our focus is on what all these contexts
have in common, and we shall try to make clear exactly where generalities must give way to context-specific proofs.   For applications, we want $\sD$ to be as concrete
as possible, something given or constructed in a way that makes it potentially useful for calculation rather than
just theory.  Towards this end, we find it essential to work with given
enrichments in naturally occurring categories $\sV$, rather than modifying $\sV$ for
greater theoretical convenience. 

{The reader is assumed to be familiar with basic model category theory, as in \cite{Hirsch, Hovey, MP}.  The last of these is the most
recent textbook souce. It was written at the same time as the first draft of this paper, which can be viewed as a natural sequel to the basics
of enriched model category theory as presented there.  We give full details or precise references on everything we use that is not in \cite{MP}.}

It is a pleasure to thank an anonymous referee for an especially helpful report.
This  work  was  partially  supported  by Simons Collaboration Grant No. 282316 held by the first author.

\section{Comparisons between model categories $\sM$ and $\mathbf{Pre} (\sD,\sV)$}\label{comparisons}

\subsection{Standing assumptions on $\sV$, $\sM$, and $\sD$}\label{standass}

We fix assumptions here.  We fill in background and comment on our 
choices of assumptions and notations in \S\ref{enriched} and \S\ref{CatPre}.

Throughout this paper, $\sV$ will be a bicomplete closed symmetric monoidal category that is also
a cofibrantly generated and proper monoidal model category (as specified in \cite[4.2.6]{Hovey},
 \cite[11.1.2]{Hirsch}, or \cite[16.4.7]{MP}; see \myref{monoidal} below). While it is sensible 
to require $\sV$ to be proper, we shall not make essential use of that assumption in this paper.
We write $V\otimes W$ or $V\otimes_{\sV}W$ for the product and $\ul{\sV}(V,W)$ for the internal 
hom in $\sV$, and we write $\sV(V,W)$ for the set of morphisms $V\rtarr W$ in $\sV$.  We let
$\mathbf{I}$ denote the unit object of $\sV$. We do not assume that $\mathbf{I}$ is cofibrant,
and we do not assume the monoid axiom (see \myref{monax}).
We assume given canonical sets $\mathcal{I}$ and $\mathcal{J}$ of generating cofibrations and
generating acyclic cofibrations for $\sV$.

We assume familiarity with the definitions of enriched categories, enriched functors,
and enriched natural transformations \cite{BorII, Kelly}. A brief elementary account
is given in \cite[Ch. 16]{MP} and we give some review in \S\ref{enriched} and \S\ref{CatPre}.  
We refer to these as $\sV$-categories, $\sV$-functors, and $\sV$-natural transformations. 

Throughout this paper, $\sM$ will be a bicomplete 
$\sV$-category. We explain the bicompleteness assumption in \S\ref{EnHyp}. 
We let 
$\ul{\sM}(M,N)$ denote the enriched hom object in $\sV$ between objects $M$ and $N$ of 
$\sM$.  We write $\ul{\sM}_{\sV}(M,N)$ when considering changes of enriching category.   
We write $\sM(M,N)$ for the set of morphisms $M\rtarr N$ in the underlying category of $\sM$. 
By definition,
\begin{equation}\label{homset}
\sM(M,N) = \sV(\mathbf{I},\ul{\sM}(M,N)).
\end{equation}
Bicompleteness includes having tensors and cotensors, which we denote by
\[ M\odot V \ \ \text{and} \ \  F(V,M) \]
for $M\in \sM$ and $V\in \sV$;   (\ref{biten}) gives the defining adjunctions for these objects of $\sM$.

We regard the underlying category as part of the structure of $\sM$. Philosophically, if
we think of the underlying category as the primary structure, we think of ``enriched''
as an adjective modifying the term category. If we think of the entire structure as 
fundamental, we think of ``enriched category'' as a noun (see \cite{MP} and \myref{underly}).

In fact, when thinking of it as a noun, it can sometimes be helpful to think of the underlying category 
as implicit and unimportant.  One can then think of the enrichment as specifying a 
$\sV$-category, with morphism objects $\ul{\sM}(M,N)$ in $\sV$, unit maps $\mathbf{I}\rtarr \ul{\sM}(N,N)$
in $\sV$, and a unital and associative composition law in $\sV$, but with no mention of underlying 
maps despite their implicit definition in (\ref{homset}).

We fix a small $\sV$-category $\sD$.  We then have the category $\mathbf{Pre} (\sD,\sV)$ of 
$\sV$-functors $X\colon \sD^{op}\rtarr \sV$ and $\sV$-natural transformations; we
call $X$ an enriched presheaf.  

\begin{rem}\mylabel{polemic} When considering the domain categories $\sD$ of presheaf categories, 
we are never interested in the underlying category of $\sD$ and in fact the underlying category
is best ignored.  We therefore use the notation $\sD(d,e)$ rather than $\ul{\sD}(d,e)$ for the 
hom objects in $\sV$ of the domain categories of presheaf categories.  We may issue reminders, 
but the reader should always remember this standing convention.
\end{rem}

We write $X_d$ for the object of $\sV$ that $X$ assigns to an object $d$ of $\sD$.  Then $X$ is given
by maps 
\[ X(d,e) \colon \sD(d,e) \rtarr \ul{\sV}(X_e,X_d) \]
in $\sV$.  Maps $f\colon X\rtarr Y$ of presheaves are given by maps $f_d\colon X_d\rtarr Y_d$
in $\sV$ that make the appropriate diagrams commute; see (\ref{Vnat}).  
 
\begin{rem}\mylabel{polemic2}  As we explain in \S5.1,  $\mathbf{Pre} (\sD,\sV)$
is itself the underlying category of a $\sV$-category.  We write
$\ul{\mathbf{Pre}}(\sD,\sV)(M,N)$
for the hom object in $\sV$  of morphisms of presheaves $M\rtarr N$.  It is an 
equalizer displayed in greater generality in (\ref{PXY}).  
\end{rem}

The Yoneda embedding $\bY\colon \sD\rtarr \mathbf{Pre} (\sD,\sV)$ plays an important role in the theory.

\begin{defn}\mylabel{YonNot} For $d\in \sD$, $\bY(d)$ denotes the presheaf in $\sV$ 
represented by $d$, so that $$\bY(d)_e = {\sD}(e,d);$$ 
$\bY$ is the object function of a $\sV$-functor $\bY\colon \sD\rtarr \mathbf{Pre} (\sD,\sV)$. 
Thus
$$ \bY \colon \sD(d, d') \rtarr \ul{\mathbf{Pre}}(\sD,\sV)(\bY(d),\bY(d')) $$
is a map in $\sV$ for each pair of objects $d, d'$ of $\sD$.
\end{defn}

The classical Yoneda lemma generalizes to an enriched Yoneda lemma \cite[6.3.5]{BorII}
identifying enriched natural transformations out of represented enriched functors. 
We have defined $\mathbf{Pre}(\sD,\sV)$, but we need notation for more general functor categories.

\begin{defn}  Let $\mathbf{Fun} (\sD^{op},\sM)$ denote the category of $\sV$-functors 
$\sD^{op}\rtarr \sM$ and $\sV$-natural transformations.  In particular, taking 
$\sM= \sV$,
$$\mathbf{Fun} (\sD^{op},\sV) = \mathbf{Pre} (\sD,\sV).$$
Again, as we explain in \S\ref{CatPre}, $\mathbf{Fun}(\sD^{op},\sM)$ is bicomplete and is the underlying 
category of a $\sV$-category, with hom objects displayed as equalizers in (\ref{PXY}).
\end{defn}

\begin{defn}\mylabel{Fd} Let $\text{ev}_d\colon \mathbf{Fun}(\sD^{op},\sM)\rtarr \sM$ denote the $d^{th}$ object 
$\sV$-functor, which sends $X$ to $X_d$.  Let $F_d\colon \sM\rtarr \mathbf{Fun} (\sD^{op},\sM)$ be the 
$\sV$-functor defined on objects by $F_dM = M\odot \bY(d)$, so that
\[ (F_dM)_e = M\odot \sD(e,d). \]
\end{defn}

We discuss $\sV$-adjunctions in \S\ref{EnHyp} and explain the following result in 
\S\ref{CatPre}.

\begin{prop}\mylabel{Fb} The pair $(F_d,\text{ev}_d)$ is a $\sV$-adjunction
between $\sM$ and $\mathbf{Fun} (\sD^{op},\sM)$.
\end{prop}

\begin{rem}\mylabel{Gb} Dually, we have the $\sV$-functor $G_d\colon \sM\rtarr \mathbf{Fun} (\sD,\sM)$
defined by $G_dM = F(\bY(d),M)$, and
$(\text{ev}_d,G_d)$ is a $\sV$-adjunction between $\sM$ and $\mathbf{Fun} (\sD,\sM)$. 
\end{rem}

\subsection{The categorical context for the comparisons}

Under mild assumptions, discussed in \S\ref{ModelD}, the levelwise weak equivalences 
and fibrations determine a model structure on $\mathbf{Pre} (\sD,\sV)$.  {This is usually 
verified by \myref{level2}, and it often holds for any $\sD$ by \myref{monoid}.  We assume throughout 
that all of our presheaf categories $\mathbf{Pre} (\sD,\sV)$ are such model categories.
Presheaf model categories of this sort are the starting point for a great deal of work in 
many directions.  In particular, they give the starting point 
for several constructions of the stable homotopy category 
and for Voevodsky's homotopical approach to algebraic geometry.} 
In these applications, the level model structure is just a step on the way towards 
the definition of a more sophisticated model structure, but we are interested 
in applications in which the level model structure is itself the one of interest.  

We have so far assumed no relationship between $\sD$ and $\sM$, and in practice
one encounters different interesting contexts.  We are especially interested in the 
restricted kind of $\sV$-categories $\sD$ that are given by full embeddings 
$\sD\subset \sM$, but we shall see in \cite{GM2, GMequiv} that it is worth working more generally 
with a fixed $\sV$-functor $\de \colon \sD\rtarr \sM$ as starting
point.   We set up the relevant formal context before returning to
model theoretic considerations.

\begin{notns}\mylabel{Dsets} We fix a small $\sV$-category $\sD$ and
a $\sV$-functor $\de \colon \sD\rtarr \sM$, writing $(\sD,\de )$ for 
the pair.  As a case of particular interest, for a fixed set $\sD$ 
(or $\sD_{\sM}$) of objects of $\sM$, we let $\sD$ 
also denote the full $\sV$-subcategory of $\sM$ with object set $\sD$, 
and we then implicitly take $\de$ to be the inclusion.  
\end{notns}

We wish to compare $\sM$ with $\mathbf{Pre} (\sD,\sV)$.  There are 
two relevant frameworks.  In one, $\sD$ is given a priori, independently of $\sM$, and
$\sM$ is defined in terms of $\sD$ and $\sV$.  In the other,
$\sM$ is given a priori and $\sD$ is defined in terms of $\sM$.  
Either way, we have a $\sV$-adjunction relating $\sM$ and $\mathbf{Pre} (\sD,\sV)$. 

\begin{defn}\mylabel{bU} Define a $\sV$-functor $\bU\colon \sM\rtarr \mathbf{Pre} (\sD,\sV)$ 
by letting $\bU(M)$ be the $\sV$-functor represented (or $\de$-represented)
by $M$, so that $\bU(M)_d = \ul{\sM}(\de d,M)$. The evaluation maps of this
presheaf are 
{\small
\[ \xymatrix@1{
 \ul{\sM}(\de e,M)\otimes {\sD}(d,e) \ar[r]^-{\id\otimes \de } & 
 \ul{\sM}(\de e,M)\otimes  \ul{\sM}(\de d,\de e) \ar[r]^-{\circ} 
 &  \ul{\sM}(\de d,M).\\} \] }
When $\de$ is a full embedding, $\bU$ extends the Yoneda embedding: $\bU\circ \de = \bY$.
\end{defn}

\begin{prop}\mylabel{bT} The $\sV$-functor $\bU$ has the left $\sV$-adjoint $\bT$
defined by $\bT X = X\odot_{\sD} \de$.
\end{prop}
\begin{proof}
This is an example of a tensor product of functors as specified in 
(\ref{FWXO}). It should be thought of as the extension of $X$ from $\sD$ to $\sM$. The $\sV$-adjunction
\[  \ul{\sM}(\bT X,M)\iso  \ul{\mathbf{Pre} }(\sD,\sV)(X,\bU M) \]
is a special case of (\ref{WMXO}). 
\end{proof} 

We will be studying when $(\bT,\bU)$ is a Quillen equivalence of model categories and we
record helpful observations about the unit $\et\colon \Id \rtarr \bU\bT$ and counit 
$\epz\colon \bT\bU\rtarr \Id$ of the adjunction $(\bT,\bU)$.  We are interested in
applying $\et$ to $X=F_dV\in \mathbf{Pre} (\sD,\sV)$ and $\epz$ to $d\in \sD$ when $\sD$ is
a full subcategory of $\sM$.  Remember that $F_dV = \bY(d)\odot V$.

\begin{lem}\mylabel{cute}  Let $d\in\sD$ and $V\in \sV$. Then
$\bT(F_dV)$ is naturally isomorphic to $\de d\odot V$. When 
evaluated at $e\in \sD$,  
\begin{equation}\label{eta}
\et\colon \bY(d)\odot V = F_dV \rtarr \bU\bT(F_dV) \iso
\bU(\de d\odot V)
\end{equation}
is the map
\[ \xymatrix{ 
\sD(e,d)\otimes V \ar[r]^-{\de\otimes\id} &  \ul{\sM}(\de e,\de d)\otimes V \ar[r]^-{\om} &
 \ul{\sM}(\de e,\de d\odot V),\\} \]
where $\om$ is the natural map of (\ref{keymap}). Therefore, if $\de \colon \sD\rtarr \sM$ 
is the inclusion of a full subcategory and $V=\mathbf{I}$, then 
$\et\colon {\sD}(e,d)\rtarr \ul{\sM}(e,d)$ is the identity map and 
$\epz\colon \bT\bU(d) = \bT\bY(d)\to d$ is an isomorphism.
\end{lem}
\begin{proof}  For the first statement, for any $M\in\sM$ we have 
\begin{eqnarray*}
\ul{\sM}(\bT(\bY(d)\odot V),M) & \iso & \ul{\mathbf{Pre} }(\sD,\sV)(\bY(d)\odot V,\bU(M))\\
& \iso & \ul{\sV}(V, \ul{\mathbf{Pre} }(\sD,\sV)(\bY(d),\bU(M)))\\
& \iso &\ul{\sV}(V,\ul{\sM}(\de d,M)) \\
&\iso & \ul{\sM}(\de d\odot V,M),
\end{eqnarray*}
by adjunction, two uses of (\ref{biten}) below, and the definition of tensors.
By the enriched Yoneda lemma, this implies $\bT(\bY(d)\odot V)\iso \de d\odot V$. 
The description of $\et$ follows by inspection, and the last statement holds
since $\om = \id$ when $V=\mathbf{I}$.
\end{proof}

\begin{rem}\mylabel{Ddealt}  There is a canonical factorization of the pair
$(\sD,\de)$.  We take $\sD\!_{\sM}$ to be the full $\sV$-subcategory of
$\sM$ with objects the $\de d$.  Then $\de$ factors as the composite of
a $\sV$-functor $\de \colon \sD\rtarr \sD\!_{\sM}$ and the inclusion 
$\io \colon \sD\!_{\sM} \subset \sM$.  The $\sV$-adjunction $(\bT,\bU)$
factors as the composite of $\sV$-adjunctions
\[ \de_!\colon \mathbf{Pre} (\sD,\sV) \rightleftarrows \mathbf{Pre} (\sD_{\sM},\sV) \colon 
\de ^* \ \ \
\text{and} \ \ \ \bT\colon \mathbf{Pre} (\sD_{\sM},\sV)  
\rightleftarrows \sM\colon \bU \]
(see \myref{flesh} below). As suggested by the notation, the same $\sD$ can relate to 
different categories $\sM$. However, the composite Quillen adjunction can be a Quillen 
equivalence even though neither of the displayed Quillen adjunctions is so.  
An interesting class of examples is given in \cite{GMequiv}. 
\end{rem}

\subsection{When does $(\sD,\de)$ induce an equivalent model structure on $\sM$?}\label{DKM}

With the details of context in hand, we return to the questions in the introduction.
Letting $\sM$ be a bicomplete $\sV$-category, we repeat the first question.  Here we 
start with a model category $\mathbf{Pre} (\sD,\sV)$ of presheaves and try to create a Quillen
equivalent model structure on $\sM$. Here and in the later questions, we are interested 
in Quillen $\sV$-adjunctions and Quillen $\sV$-equivalences, as defined in \myref{Quillad}.

\begin{quest}\mylabel{quest1too} For which $\de \colon \sD\rtarr \sM$ can 
one define a $\sV$-model structure on $\sM$ such that $\sM$ is 
Quillen equivalent to $\mathbf{Pre} (\sD,\sV)$?  
\end{quest}

Perhaps more sensibly, we can first ask this question for full embeddings corresponding 
to chosen sets of objects of $\sM$ and then look for more calculable smaller
categories $\sD$, using \myref{Ddealt} to break the question into two steps.

An early topological example where \myref{quest1too} has a positive answer 
is that of $G$-spaces (Piacenza \cite{Pia}, \cite[Ch. VI]{EHCT}), which we recall
and generalize in \cite{GMequiv}.   

The general answer to \myref{quest1too} starts from a model 
structure on $\sM$ defined in terms of $\sD$, which we 
call the $\sD$-model structure.  Recall
that $(\bU M)_d = \ul{M}(\de d,M)$.   

\begin{defn}\mylabel{newmod}  { Recall our standing assumption that 
$\mathbf{Pre}(\sD,\sV)$ has the level model structure of \myref{level}, which 
specifies sets $F\mathcal{I}$ and $F\mathcal{J}$ of generating cofibrations and 
generating acyclic cofibrations.}
A map $f\colon M\rtarr N$ in $\sM$ is a $\sD$-equivalence or 
$\sD$-fibration if $\bU f_d$ is a weak 
equivalence or fibration in 
$\sV$ for all $d\in \sD$; $f$ is a $\sD$-cofibration if it 
satisfies the LLP with respect to the $\sD$-acyclic $\sD$-fibrations.  
Define $\bT F\mathcal{I}$ and 
$\bT F\mathcal{J}$ to be the sets of maps in $\sM$ obtained by applying $\bT$ to  
$F\mathcal{I}$ and $F\mathcal{J}$.
\end{defn}

{We assume familiarity with the small object argument (e.g. \cite[\S15.1]{MP}).}

\begin{thm}\mylabel{Thm1} If $\bT F\mathcal{I}$ and $\bT F\mathcal{J}$ satisfy 
the small object argument
 and $\bT F\mathcal{J}$ satisfies the acyclicity condition
for the $\sD$-equivalences, then $\sM$ is a cofibrantly generated 
$\sV$-model category under the 
$\sD$-classes of maps, and $(\bT,\bU)$ is a Quillen $\sV$-adjunction. It is a Quillen 
$\sV$-equivalence if and only if the unit map $\et\colon X\rtarr \bU\bT X$ is a weak 
equivalence in $\mathbf{Pre} (\sD,\sV)$ for all cofibrant objects $X$.
\end{thm}
\begin{proof} As in \cite[11.3.2]{Hirsch},
$\sM$ inherits its $\sV$-model structure from $\mathbf{Pre} (\sD,\sV)$, via \myref{criterion}.
Since $\bU$ creates the $\sD$-equivalences and $\sD$-fibrations in $\sM$,
$(\bT,\bU)$ is a Quillen $\sV$-adjunction.  The last statement holds by
\cite[1.3.16]{Hovey} or \cite[16.2.3]{MP}.
\end{proof} 

\begin{rem}  By adjunction, the smallness condition {required for the small object
argument} holds if the
domains of maps in $\mathcal{I}$ or $\mathcal{J}$ are small with 
respect to the maps $\ul{\sM}(\de d,A)\rtarr \ul{\sM}(\de d,X)$, 
where $A\rtarr X$ is a $\bT F \mathcal{I}$ or $\bT F \mathcal{J}$
cell object in $\sM$.  This condition is usually easy to check in
practice, and it holds in general when $\sM$ is locally presentable.  
The acyclicity condition (defined in \myref{catweak})
holds if and only if $\bU$ carries relative $\bT F\mathcal{J}$-cell 
complexes to level equivalences, so it is obvious what must be proven.
However, the details of proof can vary considerably from one
context to another.
\end{rem}

\begin{rem}\mylabel{proper} Since $\sV$ is right proper and the
right adjoints $\sM(\de d,-)$ preserve pullbacks, it is 
clear that $\sM$ is right proper.  It is not clear that $\sM$
is left proper.   Since we have assumed that $\sV$ is left proper, 
$\sM$ is left proper provided that, for a cofibration $M\rtarr N$ and a weak
equivalence $M\rtarr Q$, the maps 
\[ \ul{\sM}(\de d,M)\rtarr \ul{\sM}(\de d,N) \] are 
cofibrations in $\sV$ and the canonical maps 
\[ \ul{\sM}(\de d,N)\cup_{\ul{\sM}(\de d,M)} \ul{\sM}(\de d,Q) 
\rtarr \ul{\sM}(\de d, N\cup_M Q) \]
are weak equivalences in $\sV$.  In topological situations,
left properness can often be shown in situations where it is not
obviously to be expected; see \cite[6.5]{MMSS} or \cite[5.5.1]{MS}, for example. 
\end{rem}

\begin{rem}\mylabel{howto}  To prove that $\et\colon X \rtarr \bU\bT X$ 
is a weak equivalence when $X$ is cofibrant, one may assume that $X$
is an $F\mathcal{I}$-cell complex. When $X=F_dV$, the maps 
\[ \om\colon \ul{\sM}(e,d)\otimes V\rtarr \ul{\sM}(e,d\odot V)\] 
of (\ref{keymap}) that appear in our description of $\eta$ in \myref{cute} 
are usually quite explicit, and sometimes even isomorphisms, and one first
checks that they are weak equivalences when $V$ is the source or target 
of a map in $\mathcal{I}$. One then uses that cell complexes are built up as
(transfinite) sequential colimits of pushouts of coproducts of maps in 
$F\mathcal{I}$.  There are two considerations in play.  First, one needs $\sV$ to 
be sufficiently well behaved that the relevant colimits preserve weak equivalences. 
Second, one needs $\sM$ 
and $\sD$ to be sufficiently well behaved that the right adjoint $\bU$ 
preserves the relevant categorical colimits, at least up to weak equivalence. Formally, if
$X$ is a relevant categorical colimit, $\colim X_s$ say, then
$\et\colon X_d\rtarr \ul{\sM}(\de d,\bT X)$ factors as the composite
\[ \colim (X_s)_d\rtarr \colim \ul{\sM}(\de d,\bT X_s) \rtarr 
\ul{\sM}(\de d,\colim\bT X_s),\]
and a sensible strategy is to prove that these two maps are each weak equivalences, 
the first as a colimit of weak equivalences in $\sV$ and
the second by a preservation of colimits result for $\bU$.  Suitable 
compactness (or smallness) of the objects $d$ can reduce the problem to 
the pushout case, which can be dealt with using an appropriate
version of the gluing lemma asserting that a pushout of weak equivalences
is a weak equivalence. We prefer not to give a formal axiomatization 
since the relevant verifications can be technically quite different in 
different contexts.
\end{rem}

\subsection{When is a given model category $\sM$ equivalent to some $\mathbf{Pre} (\sD,\sV)$?}

We are more interested in the second question in the introduction, which we repeat. 
Changing focus, we now start with a given model structure on $\sM$.

\begin{quest}\mylabel{quest2too} Suppose that $\sM$ is a $\sV$-model category.  
When is $\sM$ Quillen equivalent to $\mathbf{Pre} (\sD,\sV)$, where $\sD=\sD_{\sM}$ 
is the full sub $\sV$-category of $\sM$ given by some well chosen set of objects $d\in \sM$?
\end{quest}

\begin{asses}\mylabel{assess}
Since we want $\ul{\sM}(d,e)$ to be homotopically meaningful,
we require henceforward that the objects of our full subcategory
$\sD$ be bifibrant.   {As usual, we also assume that 
$\mathbf{Pre} (\sD,\sV)$ has the level model structure of \S\ref{ModelD}.}
\end{asses}

The following invariance result helps motivate the assumption that the objects
of $\sD$ be bifibrant. 

\begin{lem}\mylabel{Cutesy} Let $\sM$ be a $\sV$-model category, let $M$
and $M'$ be cofibrant objects of $\sM$, and let $N$ and $N'$ be fibrant objects
of $\sM$.  If $\ze\colon M\rtarr M'$ and $\xi\colon N\rtarr N'$ are weak equivalences
in $\sM$, then the induced maps
\[ \ze^*\colon \ul{\sM}(M',N) \rtarr \ul{\sM}(M,N) \ \ \ \text{and}\ \ \ 
\xi_*\colon \ul{\sM}(M,N) \rtarr \ul{\sM}(M,N')  \]
are weak equivalences in $\sV$.
\end{lem}
\begin{proof} We prove the result for $\xi_*$. The proof for $\ze^*$ is dual.
Consider the functor $\ul{\sM}(M,-)$ from $\sM$ to $\sV$. By Ken Brown's 
lemma (\cite[1.1.12]{Hovey} or \cite[14.2.9]{MP}) and our assumption that $N$ and $N'$ are
fibrant, it suffices to prove that $\xi_*$ is a weak equivalence when $\xi$ is an acyclic 
fibration. If $V\rtarr W$ is a cofibration in $\sV$, then $M\odot V \rtarr M\odot W$ 
is a cofibration in $\sM$ since $M$ is cofibrant and $\sM$ is a $\sV$-model category.
Therefore the adjunction (\ref{biten}) that defines $\odot$ implies that if $\xi$ is an
acyclic fibration in $\sM$, then $\xi_*$ is an acyclic fibration in $\sV$ and thus a weak 
equivalence in $\sV$.
\end{proof} 

\myref{quest2too} does not seem to have been asked before in quite this form 
and level of generality.  Working simplicially, Dugger \cite{Dug} studied a related question, 
asking when a given model category is Quillen equivalent to some 
localization of a presheaf category.  He called such an equivalence a 
``presentation'' of a model category, viewing the localization as specifying 
the relations.  That is an interesting point of view for theoretical purposes, 
since the result can be used to deduce formal properties of $\sM$ from formal 
properties of presheaf categories and localization.  However, the relevant domain 
categories $\sD$ are not intended to be small and calculationally accessible.

Working simplicially with stable model categories enriched over symmetric
spectra, Schwede and Shipley made an extensive study of essentially
this question in a series of papers, starting with \cite{SS}. The question
is much simpler to answer stably than in general, and we shall return to 
this in \S\ref{Stablesec}. 

{Of course, if the given model structure on $\sM$ is a $\sD$-model 
structure, as in \myref{Thm1}, then nothing more need be said. However, when that is not the case,
 the answer is not obvious.}  We
offer a general approach to the question.  The following starting point is immediate from 
the definitions and \myref{assess}. 

\begin{prop}\mylabel{Quill} $(\bT,\bU)$ is a Quillen adjunction between 
the $\sV$-model categories $\sM$ and $\mathbf{Pre} (\sD,\sV)$.
\end{prop}
\begin{proof} Applied to the cofibrations $\emptyset \rtarr d$ given by our assumption
that the objects of $\sD$ are cofibrant, the definition 
of a $\sV$-model structure implies that if $p\colon E\rtarr B$ is a fibration or acyclic 
fibration in $\sM$, then 
$p_*\colon \ul{\sM}(d,E)\rtarr \ul{\sM}(d,B)$ 
is a fibration or acyclic fibration in $\sV$. 
\end{proof}

{As with any Quillen adjunction, $(\bT,\bU)$ is a Quillen equivalence if and only if it induces
an adjoint equivalence of homotopy categories.  Clearly, we cannot expect this to hold unless
the  $\sD$-equivalences are closely related to the class $\sW$ of weak 
equivalences in the given model structure on $\sM$. }

\begin{defn}\mylabel{refcre} Let $\sD$ be a set of objects of $\sM$ satisfying \myref{assess}.
\begin{enumerate}[(i)]
\item Say that $\sD$ is a reflecting set if $\bU$ reflects weak equivalences
between fibrant objects of $\sM$; this means that if $M$ and $N$ are fibrant
and $f\colon M\rtarr N$ is a map in $\sM$ such that $\bU f$ is a weak
equivalence, then $f$ is a weak equivalence.
\item Say that $\sD$ is a creating set if $\bU$ creates
the weak equivalences in $\sM$; this means that a map $f\colon M\rtarr N$
in $\sM$ is a weak equivalence if and only if $\bU f$ is a weak equivalence,
so that $\sW$ coincides with the $\sD$-equivalences.
\end{enumerate}
\end{defn}

\begin{rem} Since the functor $\bU$ preserves acyclic fibrations between
fibrant objects, it preserves weak equivalences between fibrant objects
(\cite[1.1.12]{Hovey} or \cite[14.2.9]{MP}).  Therefore, if $\sD$ is a reflecting set, then 
$\bU$ creates the weak equivalences between the fibrant objects of $\sM$.
\end{rem}

Observe that \myref{Thm1} requires $\sD$ to be a creating set.  However, 
when one starts with a given model structure 
on $\sM$, there are many examples where no reasonably small set $\sD$
creates all of the weak equivalences in 
$\sM$, rather than just those between fibrant objects. On the other hand, 
in many algebraic and topological situations all objects are fibrant, and then there is 
no distinction. By \cite[1.3.16]{Hovey} or \cite[16.2.3]{MP}, we have the 
following criteria for $(\bT,\bU)$ to be a Quillen equivalence.

\begin{thm}\mylabel{Quill2}  Let $\sM$ be a $\sV$-model category and $\sD\subset \sM$
be a small full subcategory such that \myref{assess} are satisfied.
\begin{enumerate}[(i)] 
\item $(\bT,\bU)$ is a Quillen equivalence if and only if $\sD$ is a 
reflecting set and the composite
\[ \xymatrix@1{
X\ar[r]^-{\et} & \bU\bT X \ar[r]^-{\bU \la} & \bU \bR \bT X\\}\]
is a weak equivalence in $\mathbf{Pre} (\sD,\sV)$ for every cofibrant object $X$.
Here $\et$ is the unit of the adjunction and $\la\colon \Id\rtarr \bR$ is a 
fibrant replacement functor in $\sM$.  
\item When $\sD$ is a creating set, $(\bT,\bU)$ is a 
Quillen equivalence if and only if the map
$\et\colon X\rtarr \bU\bT X$ is a weak equivalence for every
cofibrant $X$.
\end{enumerate}
\end{thm}

Thus $\sM$ can only be Quillen equivalent to the presheaf category $\mathbf{Pre} (\sD,\sV)$
when $\sD$ is a reflecting set. In outline, the verification of (i) or (ii) of \myref{Quill2} proceeds 
along much the same lines as in \myref{howto}, and again we see little point 
in an axiomatization.  Whether or not the conclusion holds, we have the following observation.

\begin{prop}\mylabel{helpful?} Let $\sD$ be a creating set of objects of $\sM$ such that
$\sM$ is a $\sD$-model category, as in \myref{Thm1}. Then the identity functor on 
$\sM$ is a left Quillen equivalence from the $\sD$-model structure on $\sM$ to the 
given model structure, and $(\bT,\bU)$ is a Quillen equivalence with respect to one of
these model structures if and only if it is a Quillen equivalence with respect to the other.
\end{prop}
\begin{proof}
The weak equivalences of the two model structures on $\sM$ are the same,
and since $\bT$ is a Quillen left adjoint for both model structures, the relative 
$\bT F \mathcal{J}$--cell complexes are acyclic cofibrations in both. Their
retracts give all of the $\sD$-cofibrations, but perhaps only some of the
cofibrations in the given model structure, which therefore might have
more fibrations and so also more acyclic fibrations.
\end{proof}

A general difficulty in using a composite such as that in \myref{Quill2}(i)
to prove a Quillen equivalence is that the fibrant approximation $\bR$ is almost
never a $\sV$-functor and need not behave well with respect to colimits.
The following observation is relevant (and so is Baez's joke).

\begin{rem}
In topological situations, one often encounters Quillen equivalent
model categories $\sM$ and $\sN$ with different advantageous features. 
Thus suppose that $(\bF,\bG)$ is a 
Quillen equivalence $\sM\rtarr \sN$ such that $\sM$ but not necessarily 
$\sN$ is a $\sV$-model category and every object of $\sN$ is fibrant.
Let $X$ be a cofibrant object of $\mathbf{Pre} (\sD,\sV)$, as in \myref{Quill2}(i), and 
consider the diagram
\[ \xymatrix{
X\ar[r]^-{\et} \ar[dr] & \bU\bT X \ar[r]^-{\bU \la} \ar[d]^{\bU\ze}
& \bU \bR\bT X\ar[d]^{\bU\ze}_{\htp}\\
& \bU\bG\bF\bT X \ar[r]^-{\htp}_-{\bU\bG\bF\la}
& \bU\bG\bF \bR\bT X,}\]
where $\ze$ is the unit of $(\bF,\bG)$.  The arrows labeled ${\htp}$ are
weak equivalences because $\bR\bT X$ is bifibrant in $\sM$ and
$\bG\bF R\bT X$ is fibrant in $\sM$. Therefore the top composite is a weak
equivalence, as desired, if and only if the diagonal arrow $\bU\ze\com \et$ 
is a weak equivalence.  In effect, $\bG\bF\bT X$ is a fibrant approximation 
of $\bT X$, eliminating the need to consider $\bR$.  It can happen that $\bG$ 
has better behavior on colimits than $\bR$ does, and this can simplify the 
required verifications.
\end{rem}

\subsection{Stable model categories are categories of module spectra}\label{Stablesec}

In \cite{SS}, which has the same title as this section, Schwede and Shipley define a ``spectral category'' to be a small category enriched in the category $\SI\sS$ of symmetric spectra, and they understand a ``category of module spectra'' to be a presheaf category of the form 
$\mathbf{Pre} (\sD,\SI\sS)$ for some spectral category $\sD$.  Up to notation, their context is the same as the context of our \S1 and \S2, but restricted to  $\sV=\SI\sS$.  In particular, 
they give an answer to that case of \myref{quest3}, which we repeat.

\begin{quest}\mylabel{quest3too} Suppose that $\sM$ is a $\sV$-model category,
where $\sV$ is a stable model category.  When is $\sM$ Quillen equivalent to 
$\mathbf{Pre} (\sD,\sV)$, where $\sD$ is the full $\sV$-subcategory of $\sM$ given
by some well chosen set of objects $d\in \sM$?
\end{quest}

To say that $\sV$ is stable just means that $\sV$ is pointed and that the suspension functor 
$\SI$ on $\text{Ho}\sV$ is an equivalence.  It follows 
that $\text{Ho}\sV$ is triangulated \cite[\S7.2]{Hovey}.  It also follows 
that any $\sV$-model category $\sM$ is again stable and therefore $\text{Ho}\sM$ is triangulated.  
This holds since the suspension functor $\SI$ on $\text{Ho}\sM$ is equivalent to the derived tensor with the invertible object $\SI \mathbf{I}$ of $\text{Ho}\sV$.  

We here reconsider the work of Schwede and Shipley \cite{SS} and the later related work of Dugger \cite{Dugger} 
from our perspective.  {Schwede and Shipley} start with a stable model category $\sM$. They do not assume that it is a
$\SI\sS$-model category (which they call a ``spectral model category'') {and they are not concerned
with any other enrichment that $\sM$ might have.} Under appropriate hypotheses
on $\sM$, Hovey \cite{Hovey} defined the category $\SI\sM$ of symmetric spectra in $\sM$ and proved both that it
is a $\SI\sS$-model category and that it is Quillen equivalent to $\sM$ \cite[8.11, 9.1]{Hovey}.   Under
significantly weaker hypotheses on $\sM$, Dugger \cite[5.5]{Dugger} observed that an application of his
earlier work on presentations of model categories \cite{Dug} implies that $\sM$ is Quillen equivalent
to a model category $\sN$ that satisfies the hypotheses needed for Hovey's results.  

By the main result of Schwede and Shipley, \cite[3.9.3]{SS}, when $\sM$ and hence $\sN$ has a compact set of generators (see 
\myref{compactgen} below), $\SI\sN$ is Quillen equivalent to a presheaf category $\mathbf{Pre} (\sE,\SI\sS)$ for a full 
$\SI\sS$-subcategory $\sE$ of $\SI\sN$.  Dugger proves that one can pull back the $\SI\sS$-enrichment of 
 $\SI\sN$ along the two Quillen equivalences to obtain a $\SI\sS$-model category structure on 
$\sM$ itself.  Pulling back $\sE$ gives a full $\SI\sS$-subcategory $\sD$ of $\sM$ such that $\sM$ is Quillen equivalent to 
$\mathbf{Pre} (\sD,\SI\sS)$.  In a sequel to \cite{SS}, Schwede and Shipley \cite{SS2} show that the conclusion can be 
transported along changes of $\sV$ to any of the other standard modern model categories of spectra.

{However, stable model categories $\sM$ often appear in nature as $\sV$-enriched in an appropriate stable category 
$\sV$ other than $\SI\sS$,  and we shall work from that starting point.   It is then natural to model $\sM$ by presheaves with
values in $\sV$, starting with an appropriate full $\sV$-subcategory $\sD$ of $\sM$.  We are especially interested in finding 
explicit simplified models for $\sD$. For that purpose, it is most convenient to work directly with the given enrichment on $\sM$, not on 
some enriched category Quillen equivalent to $\sM$.  That is a central point of the sequel \cite{GM2}, where we give a convenient 
presheaf model for the category of $G$-spectra when $G$ is a finite group.}

Philosophically, it seems to us that when one starts with a nice 
$\sV$-enriched model category $\sM$, there is little if any gain in switching from $\sV$ to 
$\SI\sS$ or to any other preconceived choice.  In fact, with the switch, it is not 
obvious how to compare an intrinsic $\sV$-category $\sD$ living in $\sM$ to the 
associated spectral category living in $\SI \sM$.  When $\sV$ is $\SI\sS$ itself, this point is 
addressed in \cite[A.2.4]{SS}, and it is addressed more generally in \cite{Dugger, DS}.
We shall turn to the study of comparisons of this sort in \S\S\ref{changeDM},\ref{changeDMV}. However, 
it seems sensible to avoid unnecessary comparisons by working with given enrichments whenever possible. 

This perspective allows us to avoid the particular technology of symmetric spectra, 
which is at the technical heart of \cite{SS} and \cite{Dugger}.  A price is a loss 
of generality, since we ignore the problem of how to enrich a given stable model 
category if it does not happen to come in nature with a suitable enrichment: as our
sketch above indicates, that problem is a major focus of \cite{Dugger, SS}.  A gain, perhaps,
is brevity of exposition.

In any context, as already said, working stably makes it easier to prove Quillen
equivalences.  We give a $\sV$-analogue of \cite[Thm 3.3.3(iii)]{SS} 
after some recollections about triangulated categories that explain 
how such arguments work in general.

\begin{defn}\mylabel{compactgen} Let $\sA$ be a triangulated category with coproducts.  
An object $X$ of $\sA$ is compact (or small) if the natural map 
$\oplus \sA(X,Y_i)\rtarr \sA(X,\amalg Y_i)$
is an isomorphism for every set of objects $Y_i$. A set $\sD$ of objects
generates $\sA$ if a map $f\colon X\rtarr Y$ is an isomorphism 
if and only if $f_*\colon \sA(d,X)_*\rtarr \sA(d,Y)_*$ is an isomorphism 
for all $d\in \sD$.  We write $\sA(-,-)$ and 
$\sA(-,-)_*$ for the maps and graded maps in $\sA$. We use graded maps
so that generating sets need not be closed under $\SI$. We say that $\sD$ 
is compact if each $d\in \sD$ is compact.
\end{defn}

We emphasize the distinction between generating sets in triangulated
categories and the sets of domains (or cofibers) of generating sets 
of cofibrations in model categories.  The former generating sets can 
be much smaller. For example, in a good model category of spectra, one 
must use all spheres $S^n$ to obtain a generating set of cofibrations, 
but a generating set for the homotopy category need only contain 
$S=S^0$.  The difference is much more striking for parametrized 
spectra \cite[13.1.16]{MS}.
  
The following result is due to Neeman \cite[3.2]{Nee}. Recall that a 
localizing subcategory of a triangulated category is a sub triangulated
category that is closed under coproducts; it is necessarily also 
closed under isomorphisms.

\begin{lem}\mylabel{smallest} 
The smallest localizing subcategory
of $\sA$ that contains a compact generating set $\sD$ 
is $\sA$ itself.
\end{lem}

This result is used in tandem with the following one to prove equivalences.

\begin{lem}\label{wrap}
Let $E, F\colon \sA\rtarr \sB$ be exact and coproduct-preserving 
functors between triangulated categories and let $\ph\colon E\rtarr F$ 
be a natural transformation that commutes with $\SI$. Then the full 
subcategory of $\sA$ consisting of those objects $X$ for which $\ph$ is 
an isomorphism is localizing.
\end{lem}

When proving adjoint equivalences, the exact and coproduct-preserving 
hypotheses in the previous result are dealt with using the following 
observations (see \cite[3.9 and 5.1]{Nee2} and \cite[7.4]{FHM}). Of
course, a left adjoint obviously preserves coproducts.

\begin{lem}\mylabel{preserves} 
Let $(L,R)$ be an adjunction between triangulated
categories $\sA$ and $\sB$.  Then $L$ is exact if and only if 
$R$ is exact.  Assume that $L$ is additive and $\sA$ has a compact 
set of generators $\sD$. If $R$ preserves coproducts, then $L$ preserves 
compact objects.  Conversely, if $L(d)$ is compact for $d\in \sD$, then $R$ preserves coproducts.
\end{lem}

Returning to our model theoretic context, let $\sD$ be any small $\sV$-category, not necessarily related to
any given $\sM$.  To apply the results above, we need a compact generating 
set in $\text{Ho}\mathbf{Pre} (\sD,\sV)$, and for that we need a compact generating 
set in $\text{Ho}\sV$.  It is often the case in applications that the unit object $\mathbf{I}$ 
is itself a compact generating set, but it is harmless to start out more generally.  
We have in mind equivariant applications where that would fail.

\begin{lem}\mylabel{complem} Let $\text{Ho}\sV$ have a compact generating set 
$\sC$ and define $F\sC$ to be the set of objects 
$F_dc\in\text{Ho}\mathbf{Pre} (\sD,\sV)$, where $c\in\sC$ and $d\in\sD$.  Assume either that
cofibrant presheaves are levelwise cofibrant or that any coproduct of weak 
equivalences in $\sV$ is a weak equivalence. Then $F\sC$ is a compact generating set. 
\end{lem}
\begin{proof} Since this is a statement about homotopy categories, we may assume without 
loss of generality that each $c\in\sC$ is cofibrant in $\sV$.
Since the weak equivalences and fibrations in $\mathbf{Pre} (\sD,\sV)$ are defined levelwise,
they are preserved by $\text{ev}_d$. Therefore $(F_d,\text{ev}_d)$ is a Quillen adjunction,
hence the adjunction passes to homotopy categories.  Since coproducts
in $\mathbf{Pre} (\sD,\sV)$ are defined levelwise, they commute with $\text{ev}_d$. Therefore  
the map
\[ \oplus_i \, \text{Ho}\mathbf{Pre} (\sD,\sV)(F_dc,Y_i)\rtarr \text{Ho}\mathbf{Pre} (\sD,\sV)(F_dc,\amalg_i \, Y_i)\]
can be identified by adjunction with the isomorphism 
\[ \oplus_i \, \text{Ho}\sV(c, \text{ev}_d Y_i)\rtarr \text{Ho}\sV(c,\amalg_i \, \text{ev}_d Y_i),\]
where the $Y_i$ are bifibrant presheaves. The identification of sources is immediate. For 
the identification of targets, either of our alternative assumptions ensures that the coproduct 
$\amalg \text{ev}_d Y_i$
in $\sV$ represents the derived coproduct $\amalg \text{ev}_d Y_i$ in $\text{Ho}\sV$.
Since the functors $\text{ev}_d$ create the weak 
equivalences in $\mathbf{Pre} (\sD,\sV)$, it is also clear by adjunction that $F\sC$ 
generates $\text{Ho}\mathbf{Pre} (\sD,\sV)$ since $\sC$ generates $\text{Ho}\sV$.
\end{proof}

By \myref{Fb}, if $\sC = \{\mathbf{I}\}$, then $F\sC$ can be identified with $\{\bY(d)\}$.
Switching context from the previous section by replacing reflecting sets by
generating sets, we have the following result. When $\sV$ is the category of symmetric
spectra, it is Schwede and Shipley's result \cite[3.9.3(iii)]{SS}.  We emphasize for
use in the sequel \cite{GM2} that our general version can apply even when $\mathbf{I}$ is not 
cofibrant and $\sV$ does not satisfy the monoid axiom. We fix a cofibrant approximation 
$\bQ\mathbf{I}\rtarr \mathbf{I}$. 

\begin{thm}\mylabel{Stablethm} Let $\sM$ be a $\sV$-model category, where
$\sV$ is stable and $\{\mathbf{I}\}$ is a compact generating set in $\text{Ho}\sV$.
Let $\sD$ be a full $\sV$-subcategory of bifibrant objects of $\sM$ such that 
$\mathbf{Pre} (\sD,\sV)$ is a model category and the set of objects of $\sD$ 
is a compact generating set in $\text{Ho}\sM$.
Assume the following two conditions.
\begin{enumerate}[(i)]
\item Either $\mathbf{I}$ is cofibrant in $\sV$ or every object
of $\sM$ is fibrant and the induced map $F_d\bQ{\mathbf{I}} \rtarr F_d\mathbf{I}$ 
is a weak equivalence for each $d\in\sD$. 
\item Either cofibrant presheaves are level cofibrant or
coproducts of weak equivalences in $\sV$ are weak equivalences.
\end{enumerate}
Then $(\bT,\bU)$ is a Quillen equivalence between $\mathbf{Pre} (\sD,\sV)$ 
and $\sM$. 
\end{thm}
\begin{proof}  In view of what we have already proven, it only remains
to show that the derived adjunction $(\bT,\bU)$ on homotopy categories 
is an adjoint equivalence. The distinguished triangles in $\text{Ho}\sM$ and 
$\text{Ho}\mathbf{Pre} (\sD,\sV)$ are generated by the cofibrations in the underlying model
categories.  Since $\bT$ preserves cofibrations, its derived functor is 
exact, and so is the derived functor of $\bU$.  We claim that \myref{preserves}
applies to show that $\bU$ preserves coproducts. By \myref{complem} and hypothesis,
$\{F_d\mathbf{I}\}$ is a compact set of generators for $\text{Ho}\mathbf{Pre} (\sD,\sV)$. To prove
the claim, we 
must show that $\{\bT F_d\mathbf{I}\}$ is a compact set of generators for $\text{Ho}\sM$. 
It suffices to show that $\bT F_d \mathbf{I} \iso d$ in $\text{Ho}\sM$, and
\myref{cute} gives that $\bT F_d \mathbf{I} \iso d$ in $\sM$. If $\mathbf{I}$ is 
cofibrant, this is an isomorphism between cofibrant objects of $\sM$. If not, 
the unit axiom for the $\sV$-model category $\sM$ gives that the induced map 
$d\odot \bQ\mathbf{I}\rtarr d\odot \mathbf{I}\iso d$ is a weak equivalence for $d\in \sD$.
Since $\bT F_d V \iso d\odot V$ for $V\in\sV$, this is a weak equivalence
$\bT F_d Q\mathbf{I}\rtarr \bT F_d \mathbf{I}$.
Either way, we have the required isomorphism in $\text{Ho}\sM$.   

Now, in view of Lemmas \ref{smallest}, \ref{wrap}, and \ref{complem}, we need only show that 
the isomorphisms $\et\colon F_d\mathbf{I}\rtarr \bU\bT F_d \mathbf{I}$ in $\mathbf{Pre} (\sD,\sV)$ and 
$\epz\colon \bT\bU d\rtarr d$ in $\sM$ given in \myref{cute} imply that their derived
maps are isomorphisms in the respective homotopy categories $\text{Ho}\mathbf{Pre} (\sD,\sV)$ and 
$\text{Ho}\sM$. Assume first that $\mathbf{I}$ is cofibrant. Then the former implication
is immediate and, since $\bU(d) = F_d(\mathbf{I})$ is cofibrant, so is the latter.

Thus assume that $\mathbf{I}$ is not cofibrant. Then to obtain $\et$ on the homotopy category
$\text{Ho}\mathbf{Pre} (\sD,\sV)$, we must replace $\mathbf{I}$ by $Q\mathbf{I}$ before applying the map 
$\et$ in $\sV$. By (\ref{eta}), when we apply $\et\colon \mathrm{Id}\rtarr \bU\bT$ 
to $F_dV$ for $V\in \sV$ and evaluate at $e$, we get a natural map 
\[ \xymatrix{ 
\et\colon {\sD}(e,d)\otimes V =\ul{\sM}(e,d)\otimes V 
\ar[r]  & \ul{\sM}(e,d\odot V)  \\} \]
that is an isomorphism when $V= \mathbf{I}$. We must show that it is a weak equivalence
when $V=\bQ\mathbf{I}$. To see this, observe that we have a commutative square 
\[ \xymatrix{
\ul{\sM}(e,d)\otimes Q\mathbf{I}  \ar[r]^-{\et} \ar[d]  & \ul{\sM}(e,d\odot Q\mathbf{I}) \ar[d]\\
\ul{\sM}(e,d)\otimes \mathbf{I}  \ar[r]_-{\et}   & \ul{\sM}(e,d\odot \mathbf{I})\\} \]
The left vertical arrow is a weak equivalence by assumption.
The right vertical arrow is a weak equivalence by \myref{Cutesy} and our assumption that 
all objects of $\sM$ are fibrant. 
Therefore $\et$ is a weak equivalence when $V=\bQ\mathbf{I}$. Similarly, to pass to the homotopy
category $\text{Ho}\sM$, we must replace $\bU(d) = F_d(\mathbf{I})$ by a cofibrant approximation
before applying $\epz$ in $\sM$. By assumption, $F_d\bQ{\mathbf{I}} \rtarr F_d\mathbf{I}$ is such 
a cofibrant approximation.  Up to isomorphism, $\bT$ takes this map to the weak equivalence
$d\odot \bQ{\mathbf{I}}\rtarr d\odot \mathbf{I}\iso d$, and the conclusion follows. 
\end{proof}

\begin{rem}\mylabel{usehow}   As discussed in \S\ref{How?}, it is possible that \myref{Muro} below can be used to replace $\sV$ by a
Quillen equivalent model category $\tilde{\sV}$  in which $\mathbf{I}$ is cofibrant, so that (i) holds automatically.
\end{rem}

\begin{rem} Since the functor $F_d$ is strong symmetric monoidal, the assumption that 
$F_d\bQ{\mathbf{I}} \rtarr F_d\mathbf{I}$ 
is a weak equivalence says that $(F_d,\text{ev}_d)$ is a monoidal Quillen adjunction in the sense
of \myref{monQuill} below. The assumption holds by the unit axiom for the $\sV$-model category 
$\sM$ if the objects $\sD(d,e)$ are cofibrant in $\sV$. 
\end{rem}

\begin{rem}\mylabel{addend}  More generally, if $\text{Ho}\sV$ 
has a compact generating set $\sC$, then \myref{Stablethm} will hold 
as stated provided that $\et\colon F_dc\rtarr \bU\bT F_d c$ 
is an isomorphism in $\text{Ho}\mathbf{Pre} (\sD,\sV)$ for all $c\in \sC$.
\end{rem}

\begin{rem} When $\sM$ has both the given model structure and the $\sD$-model
structure as in \myref{Thm1}, where the objects of $\sD$ form a creating set in $\sM$, 
then the identity functor of $\sM$ is a Quillen equivalence from the $\sD$-model 
structure to the given model structure on $\sM$, by \myref{helpful?}. In practice, 
the creating set hypothesis never applies when working in a simplicial 
context, but it can apply when working in topological or homological 
contexts.
\end{rem}

Thus the crux of the answer to \myref{quest3too} about stable model categories is to identify 
appropriate compact generating sets in $\sM$.  The utility of the 
answer depends on understanding the associated hom objects, with their
composition, in $\sV$.

\section{Changing the categories $\sD$ and $\sM$, keeping $\sV$ fixed}\label{changeDM}

We return to the general theory and consider when we can change $\sD$, keeping $\sV$
fixed, without changing the Quillen equivalence class of $\mathbf{Pre} (\sD,\sV)$. This is 
crucial to the sequel \cite{GM2}.  We allow $\sV$
also to change in the next section.  Together with our standing assumptions on $\sV$ and $\sM$ 
from \S\ref{standass}, we assume once and for all that all categories in this section and the 
next satisfy the hypotheses of \myref{level2}. This ensures that all of our presheaf categories 
$\mathbf{Pre} (\sD,\sV)$, and  $\mathbf{Fun} (\sD^{op},\sM)$ are cofibrantly generated $\sV$-model categories.  
We will not repeat this standing assumption.

\subsection{Changing $\sD$}

In applications, especially in the sequel \cite{GM2}, we are especially interested in changing a given diagram category 
$\sD$ to a more calculable equivalent. We might also be interested in changing the 
$\sV$-category $\sM$ to a Quillen equivalent $\sV$-category $\sN$, with $\sD$ fixed, 
but the way that change works is evident from our levelwise definitions.

\begin{prop}\mylabel{fleshtoo} For a
$\sV$-functor $\xi\colon \sM\rtarr \sN$ and any small $\sV$-category $\sD$, there is an induced
$\sV$-functor $\xi_*\colon \mathbf{Fun} (\sD^{op},\sM)\rtarr \mathbf{Fun} (\sD^{op},\sN)$, and it induces
an equivalence of homotopy categories if $\xi$ does so.  A
Quillen adjunction or Quillen equivalence between $\sM$ and $\sN$
induces a Quillen adjunction or Quillen equivalence between
$\mathbf{Fun} (\sD^{op},\sM)$ and $\mathbf{Fun} (\sD^{op},\sN)$. 
\end{prop}

We have several easy observations about changing $\sD$, with $\sM$ fixed. 
Before returning to model categories, we record a categorical observation.  In the
rest of this section, $\sM$ is any $\sV$-category, but our main interest is 
in the case $\sM=\sV$.

\begin{lem}\mylabel{fleshcat} Let $\nu\colon \sD\rtarr \sE$  be a $\sV$-functor
and $\sM$ be a $\sV$-category. Then there is a $\sV$-adjunction $(\nu_!,\nu^*)$ 
between $\mathbf{Fun} (\sD^{op},\sM)$ and $\mathbf{Fun} (\sE^{op},\sM)$. 
\end{lem}
\begin{proof}
 The $\sV$-functor $\nu^*$ restricts a presheaf $Y$ on $\sE$ 
to the presheaf $Y\com \nu$ on $\sD$. Its left adjoint $\nu_!$ sends a presheaf 
$X$ on $\sD$ to its left Kan extension, or prolongation, along $\nu$ 
(e.g. \cite[23.1]{MMSS}).  Explicitly, $(\nu_!X)_e = X\otimes_{\sD} \nu_e$, 
where $\nu_e\colon \sD\rtarr \sV$ is given on objects by 
$\nu_e(d) = {\sE}(e,\nu d)$ and on hom objects by the adjoints of the composites
\[ \xymatrix@1{
{\sD}(d,d')\otimes {\sE}(e,\nu d) \ar[r]^-{\nu\otimes\id} &
{\sE}(\nu d, \nu d')\otimes {\sE}(e,\nu d) \ar[r]^-{\com} &
{\sE}(e,\nu d').\\} \]
The tensor product of functors is recalled in (\ref{FWXO}).
\end{proof}

\begin{defn}\mylabel{Nuisance} Let $\nu\colon \sD\rtarr \sE$ be a $\sV$-functor
and let $\sM$ be a $\sV$-model category. 
\begin{enumerate}[(i)]
\item $\nu$ is weakly full and faithful if
each $\nu\colon \sD(d,d')\rtarr \sE(\nu d,\nu d')$ is a weak equivalence in $\sV$.
\item $\nu$ is essentially surjective if each object $e\in \sE$ is isomorphic 
(in the underlying category of $\sE$) to an object $\nu d$ for some $d\in\sD$.
\item $\nu$ is a weak equivalence if it is weakly full and faithful and 
essentially surjective.
\item $\nu$ is an $\sM$-weak equivalence if 
\[ \nu\odot \id\colon {\sD}(d,d')\odot M\rtarr {\sE}(\nu d,\nu d')\odot M\]
is a weak equivalence in $\sM$ for all cofibrant $M$ and 
$\nu$ is essentially surjective.
\end{enumerate}
\end{defn}

\begin{prop}\mylabel{flesh} Let $\nu\colon \sD\rtarr \sE$ be a $\sV$-functor
and let $\sM$ be a $\sV$-model category. Then $(\nu_!,\nu^*)$ is a Quillen adjunction, 
and it is a Quillen equivalence if $\nu$ is an $\sM$-weak equivalence.  
\end{prop}

\begin{proof}  We have a Quillen adjunction since $\nu^*$ preserves (level) fibrations and weak equivalences. By 
 \cite[1.3.16]{Hovey} or \cite[16.2.3]{MP}, to show that $(\nu_!,\nu^*)$ is a Quillen equivalence, it suffices to show
 that $\nu^*$ creates the weak equivalences in $\mathbf{Fun} (\sD^{op},\sM)$ and that $\et\colon X \rtarr \nu^*\nu_!X$ 
 is a weak equivalence when $X$ is cofibrant.  When $\nu$ is essentially surjective, easy diagram chases show 
that $\nu^*$ creates the fibrations and weak equivalences of $\mathbf{Fun} (\sD^{op},\sM)$. Comparing composites
of left adjoints, $\nu_! F_d$ is the left adjoint $F_{\nu d}$ of $ev_d\com \nu^*$, and
$\et\colon X\rtarr \nu^*\nu_! X$ is given on objects $X=F_dM$ by maps of the form that we require to be weak 
equivalences when $\nu$ is an $\sM$-weak equivalence.  The functor $\nu^*$ 
preserves colimits, since these are defined levelwise, and the relevant colimits 
(those used to construct cell objects) preserve weak equivalences.  Therefore $\et$ 
is a weak equivalence when $X$ is cofibrant.
\end{proof}

\begin{rem}\mylabel{nunu} Let $\sD \subset \sE$ be sets of bifibrant 
objects in $\sM$ and let $\nu\colon \sD\rtarr \sE$ be
the corresponding inclusion of full $\sV$-subcategories of $\sM$.
If $\sD$ is a reflecting or creating set of objects in the
sense of \myref{refcre} or if $\sD$ is a generating set in the sense 
of \myref{compactgen}, then so is $\sE$. Therefore, if \myref{Quill2} or 
\myref{Stablethm} applies to prove that $\bU\colon \sM\rtarr \mathbf{Pre} (\sD,\sV)$ is a right
Quillen equivalence, then it also applies to prove that 
$\bU\colon \sM\rtarr \mathbf{Pre} (\sE,\sV)$ is a right Quillen equivalence.  Since
$\nu^*\bU = \bU$, this implies that $\nu^*\colon \mathbf{Pre} (\sE,\sV)\rtarr \mathbf{Pre} (\sD,\sV)$
is a Quillen equivalence, even though the ``essentially
surjective'' hypothesis in \myref{flesh} generally fails in this situation.
\end{rem}

\subsection{Quasi-equivalences and changes of $\sD$}\label{QEchangesec}

Here we describe a Morita type criterion for when two $\sV$-categories $\sD$ and $\sE$ are 
connected by a zigzag of weak equivalences. This generalizes work along the same lines of Keller 
\cite{Keller}, Schwede and Shipley \cite{SS}, and Dugger \cite{Dugger}, which deal with  particular 
enriching categories, and we make no claim to originality. It can be used in tandem with 
\myref{flesh} to obtain zigzags of weak equivalences between categories of presheaves.

Recall (cf. \S\ref{EnHyp}) that we have the $\sV$-product $\sD^{\text{op}}\otimes \sE$ between the 
$\sV$-categories $\sD^{\text{op}}$ and $\sE$. 
The objects of $\mathbf{Pre} (\sD^{op}\otimes \sE,\sV)$ are often called ``distributors'' 
in the categorical literature, but we follow \cite{SS} and call them
$(\sD,\sE)$-bimodules. Thus a $(\sD,\sE)$-bimodule $\sF$ is a contravariant
$\sV$-functor $\sD^{\text{op}}\otimes \sE\rtarr \sV$.  It is convenient to write
the action of $\sD$ on the left (since it is covariant) and the action of $\sE$ 
on the right. We write ${\sF}(d,e)$ for the object in $\sV$ that $\sF$ assigns to the
object $(d,e)$. The definition encodes three associativity diagrams
\[\xymatrix{ {\sD}(e,f)\otimes {\sD}(d,e)\otimes {\sF}(c,d) \ar[r] \ar[d] 
&  {\sD}(d,f)\otimes {\sF}(c,d) \ar[d]\\
{\sD}(e,f) \otimes {\sF}(c,e) \ar[r] & {\sF}(c,f) \\} \]
\[\xymatrix{ \sD(e,f)\otimes {\sF}(d,e)\otimes {\sE}(c,d) \ar[r] \ar[d] 
&  {\sF}(d,f)\otimes \sE(c,d) \ar[d]\\
{\sD}(e,f) \otimes {\sF}(c,e) \ar[r] & {\sF}(c,f) \\} \]
\[ \xymatrix{ {\sF}(e,f)\otimes \sE(d,e)\otimes {\sE}(c,d) \ar[r] \ar[d] 
&  {\sF}(d,f)\otimes {\sE}(c,d) \ar[d]\\
{\sF}(e,f) \otimes {\sE}(c,e) \ar[r] & {\sF}(c,f)\\} \]
and two unit diagrams
\[ \xymatrix{ \mathbf{I} \otimes {\sF}(c,d) \ar[r] \ar[dr]_{\iso} 
& {\sD}(d,d)\otimes {\sF}(c,d) \ar[d]\\
& {\sF}(c,d)\\} \quad
 \xymatrix{ {\sF}(d,e)\otimes \mathbf{I} \ar[r] \ar[dr]_{\iso}
& {\sF}(d,e)\otimes {\sE}(d,d) \ar[d]  \\
& {\sF}(d,e). \\} \]

The following definition and proposition are adapted from work of 
Schwede and Shipley \cite{SS}; see also \cite{Dugger}. They encode 
and exploit two further unit conditions. 

\begin{defn}\mylabel{QEquiv} Let $\sD$ and $\sE$ have the same sets of objects, denoted $\bO$. 
Define a quasi-equivalence between $\sD$ and $\sE$ to be a $(\sD,\sE)$-bimodule $\sF$
together with a map $\ze_d\colon \bf{I}\rtarr \sF$($d$,$d$) for each $d\in \bO$ such that 
for all pairs $(d,e)\in \bO$, the maps
\begin{equation}\label{QWeak} 
(\ze_d)^*\colon \sD(d,e) \rtarr \sF(d,e) \ \ \ \text{and}\ \ \ 
(\ze_e)_*\colon \sE(d,e) \rtarr \sF(d,e) 
\end{equation}
in $\sV$ given by composition with $\ze_d$ and $\ze_e$ are weak equivalences. Given $\sF$ and 
the maps $\ze_d$, define a new $\sV$-category $\sG(\sF,\ze)$
with object set $\bO$ by letting ${\sG}(\sF,\ze)(d,e)$ be the pullback in $\sV$ displayed in the diagram 
\begin{equation}\label{quasipb} \xymatrix{
{\sG}(\sF,\ze)(d,e)  \ar[r] \ar[d] & {\sE}(d,e) \ar[d]^{(\ze_e)_*} \\
{\sD}(d,e) \ar[r]_{(\ze_d)^*} & {\sF}(d,e) \\}
\end{equation}
Its units and composition are induced from those of $\sD$ and $\sE$ and the bimodule
structure on $\sF$ by use of the universal
property of pullbacks. The unlabelled arrows specify $\sV$-functors
\begin{equation}\label{quasipb2} 
\sG(\sF,\ze) \rtarr \sD \ \ \ \text{and} \ \ \ \sG(\sF,\ze)\rtarr \sE.
\end{equation}
\end{defn} 

\begin{prop}\mylabel{QuasiEquiv} Assume that the unit $\mathbf{I}$ is cofibrant
in $\sV$. If $\sD$ and $\sE$ are quasi-equivalent, then there is a chain of weak 
equivalences connecting $\sD$ and $\sE$.
\end{prop}
\begin{proof} Choose a quasi-equivalence $(\sF,\ze)$. If either all $(\ze_d)^*$
or all $(\ze_e)_*$ are acyclic fibrations, then all four arrows in (\ref{quasipb})
are weak equivalences and (\ref{quasipb2}) displays a zigzag of weak equivalences
between $\sD$ and $\sE$. We shall reduce the general case to two applications of
this special case. Observe that by taking a fibrant replacement in the category 
$\mathbf{Pre} (\sD^{op}\otimes \sE,\sV)$, we may assume without loss of generality that 
our given $(\sD,\sE)$-bimodule $\sF$ is fibrant, so that each ${\sF}(d,e)$ is 
fibrant in $\sV$. 

For fixed $e$, the adjoint of the right action of $\sE$ on $\sF$ gives maps
\[ {\sE}(d,d') \rtarr \ul{\sV}({\sF}(d',e),{\sF}(d,e)) \]
that allow us to view the functor $\bF(e)_d = {\sF}(d,e)$ as an 
object of $\mathbf{Pre} (\sE,\sV)$; it is fibrant since each ${\sF}(d,e)$ is fibrant in $\sV$. 
Fixing $e$ and letting $d$ vary, the maps $(\ze_e)_*$ 
of (\ref{QWeak}) specify a map $\bY(e) \rtarr \bF(e)$ in $\sV^{\sE}$.
By hypothesis, this map is a level weak equivalence, and it is thus a weak equivalence
in $\mathbf{Pre} (\sE,\sV)$. Factor it as the composite of an acyclic cofibration
$\io(e)\colon \bY(e)\rtarr X(e)$ and a fibration $\rh(e)\colon X(e)\rtarr \bF(e)$.
Then $\rh(e)$ is acyclic by the two out of three property. By \myref{RepCof}, our 
assumption that $\mathbf{I}$ is cofibrant implies that $\bY(e)$ and therefore
$X(e)$ is cofibrant in $\sV^{\sE}$, and $X(e)$ is fibrant since $\bF(e)$ is fibrant. 
Let $\sE\!nd(X)$ denote the full subcategory of $\mathbf{Pre} (\sE,\sV)$ whose objects are the 
bifibrant presheaves $X(e)$.

Now use (\ref{PXY}) to define 
\[ \ul{\sY}(d,e) = \ul{\mathbf{Pre} }(\sE,\sV)(\bY(d),X(e)) \iso X(e)_d, \]
where the isomorphism is given by the enriched Yoneda lemma, and
\[ \ul{\sZ}(d,e) = \ul{\mathbf{Pre} }(\sE,\sV)(X(d),\bF(e)). \]
Composition in
$\mathbf{Pre} (\sE,\sV)$ gives a left action of $\sE\!nd(X)$ on $\sY$ and a right
action of $\sE\!nd(X)$ on $\sZ$. Evaluation 
$$\ul{\mathbf{Pre} }(\sE,\sV)(\bY(d),X(e))\otimes \bY(d) \rtarr X(e)$$ 
gives a right
action of $\sE$ on $\sY$. The action of $\sD$ on $\sF$ gives maps
$${\sD}(e,f)\rtarr \ul{\mathbf{Pre} }(\sE,\sV)(\bF(e),\bF(f)),$$ 
and these together with composition in $\mathbf{Pre} (\sE,\sV)$ give a left action of $\sD$ on $\sZ$. 
These actions make $\sY$ an $(\sE\!nd(X),\sE)$-bimodule and $\sZ$ a $(\sD,\sE\!nd(X))$-bimodule.
We may view the weak equivalences $\io(e)$ as maps 
$\io_e\colon \mathbf{I}\rtarr \ul{\sY}(e,e)$ and the weak equivalences 
$\rh(e)$ as maps $\rh_e\colon \mathbf{I}\rtarr \ul{\sZ}(e,e)$.  We claim that
$(\sY,\io)$ and $(\sZ,\rh)$ are quasi-equivalences to which the acyclic
fibration special case applies, giving a zigzag of weak equivalences
\begin{equation}\label{5zigzag}
 \xymatrix@1{
\sE & \sG(\sY,\io) \ar[l] \ar[r] & \sE\!nd(X) & \sG(\sZ,\rh) \ar[l] \ar[r] & \sD.\\}
\end{equation}
The maps
$$(\io)_*\colon \bY(e)_d ={\sE}(d,e) \rtarr \sY(d,e) = \ul{\sV}^{\sE}(\bY(d),X(e))\iso X(e)_d$$ 
are the weak equivalences $\io\colon \bY(e)_d\rtarr X(e)_d$. The maps 
\[ (\io_d)^*\colon \ul{\sV}^{\sE}(X(d),X(e)) \rtarr \ul{\sV}^{\sE}(\bY(d), X(e)) \]
are acyclic fibrations since $\io_d$ is an acyclic cofibration and $X(e)$ is fibrant.
This gives the first two weak equivalences in the zigzag (\ref{5zigzag}).  The maps  
$$(\rh_d)^*\colon \bY(e)_d={\sD}(d,e) \rtarr \sZ(d,e) = \ul{\sV}^{\sE}(X(d),\bF(e)) $$
are weak equivalences since their composites with the maps
$$ (\io_d)^*\colon \ul{\sV}^{\sE}(X(d),\bF(e)) 
\rtarr  \ul{\sV}^{\sE}(\bY(d),\bF(e))\iso \bF(e)_d$$
are the original weak equivalences $(\ze_d)^*$.
The maps 
\[ (\rh_e)_*\colon \ul{\sV}^{\sE}(X(d),X(e)) \rtarr \ul{\sV}^{\sE}(X(d),\bF(e)) \]
are acyclic fibrations since $\rh_e$ is an acyclic fibration and $X(d)$ is cofibrant. 
This gives the second two weak equivalences in the zigzag (\ref{5zigzag}).
\end{proof}

\begin{rem}\mylabel{slippery} The assumption that $\mathbf{I}$ is cofibrant is only used to ensure
that the represented presheaves $\bY(e)$ are cofibrant. If we know that in some
other way, we do not need the assumption.  Since 
the hypotheses and conclusion only involve the weak equivalences in $\sV$, not
the rest of its model structure, we can replace the model structure on $\sV$ by the 
Quillen equivalent model structure $\tilde{\sV}$ of \myref{Muro} below, in which $\mathbf{I}$ is 
cofibrant, to eliminate the assumption.  As discussed in \S\ref{How?}, this entails checking that
all presheaf categories in sight still have the level model structure of \myref{level2}, as in our
standing assumption.
\end{rem}

\subsection{Changing full subcategories of Quillen equivalent model categories}

We show here how to obtain quasi-equivalences between full subcategories
of Quillen equivalent $\sV$-model categories $\sM$ and $\sN$.   \myref{Cutesy}
implies the following invariance statement.

\begin{lem}\mylabel{CutesyToo}  Let $(\bT,\bU)$ be a Quillen $\sV$-equivalence 
between $\sV$-model categories $\sM$ and $\sN$. Let $\{M_d\}$ be 
a set of bifibrant objects of $\sM$ and $\{N_d\}$ be a set of bifibrant objects 
of $\sN$ with the same indexing set $\bO$. Suppose given weak equivalences 
$\ze_d\colon \bT M_d\rtarr N_d$ for all $d$. Let $\sD$ and $\sE$ be the full 
subcategories of $\sM$ and $\sN$ with objects $\{M_d\}$ and $\{N_d\}$. Then the 
$\sV$-categories $\sD$ and $\sE$ are quasi-equivalent.
\end{lem} 
\begin{proof} Define 
$$\sF(d,e) = \ul{\sN}(\bT M_d,N_e)\iso \ul{\sM}(M_d,\bU N_e).$$ 
Composition in $\ul{\sN}$ and $\ul{\sM}$ gives $\sF$ an $(\sE,\sD)$-bimodule
structure. The given weak equivalences $\ze_d$ are maps 
$\ze_d\colon \mathbf{I}\rtarr \sF(d,d)$,
and we also write $\ze_d$ for the adjoint weak equivalences $M_d\rtarr \bU N_d$. 
By \myref{Cutesy}, the maps
\[ (\ze_d)^*\colon \ul{\sN}(N_d,N_e) \rtarr \ul{\sN}(\bT M_d,N_e) \ \ \ \text{and}\ \ \ 
\ul{\sM}(M_d,\bU N_e) \ltarr \ul{\sM}(M_d,M_e) \colon  (\ze_e)_*  \]
are weak equivalences since the sources are cofibrant and the targets are fibrant.
\end{proof}

The case $\sM=\sN$ is of particular interest. 

\begin{cor}\mylabel{CutesyTwee} If $\{M_d\}$ and $\{N_d\}$ are two sets of
bifibrant objects of $\sM$ such that $M_d$ is weakly equivalent to $N_d$
for each $d$, then the full $\sV$-subcategories of $\sM$ with object sets
$\{M_d\}$ and $\{N_d\}$ are quasi-equivalent.
\end{cor}

Unlike \myref{QuasiEquiv}, these results do not assume that $\mathbf{I}$ is cofibrant.   
In our applications in the sequel \cite{GM2}, we can apply \myref{QuasiEquiv} to convert the
resulting quasi-equivalences to weak equivalences to which \myref{flesh} can be applied to obtain 
$\sM$-weak equivalences between functor categories. We indicate how this works and why we have 
no need for \myref{slippery} to make such applications rigorous.

\begin{rem}\mylabel{subtle} In stable homotopy theory, we encounter model categories $\sV$ and $\sV_+$ 
with the same underlying symmetric monoidal category and the same weak equivalences such that the identity functor 
$\sV_+\rtarr \sV$ is a left Quillen equivalence. The unit object $\mathbf{I}$ is cofibrant in $\sV$ but not in $\sV_+$. 
We also encounter interesting $\sV$-enriched categories $\sM$ that are $\sV_+$-model categories but 
{\em not} $\sV$-model categories. Since the weak equivalences in $\sV$ and $\sV_+$ are the same, we can apply
\myref{CutesyToo} to $\sV_+$-model categories
to obtain quasi-equivalences.
These quasi-equivalences can then be fed into 
Propositions \ref{flesh}  and \ref{QuasiEquiv}, using the model category $\sV$ with cofibrant unit.  
 \myref{Muro} (see also \myref{Vpos}) gives an intermediate model structure $\tilde{\sV}_+$ to which 
\myref{slippery} applies, but the logic of our
applications has no need for it.  
\end{rem}

\subsection{The model category $\sV\bO$-$\sC\!at$}\label{modelVO}

As a preliminary to change results for $\sV$ and $\sD$ in the next section,
we need a model category of domain $\sV$-categories 
for categories of presheaves in $\sV$.  In this section, all domain $\sV$-categories $\sD$ 
have the same set of objects $\bO$. This simplifying restriction is not essential
(compare \cite{BM, Lurie, Stan}) but is convenient for our purposes.  Let $\sV\bO$-$\sC\!at$ be the category of 
$\sV$-categories with object set $\bO$ and $\sV$-functors that are the identity 
on objects.  The following result is \cite[6.3]{SS}, and we just sketch the proof. 
Recall our standing hypothesis that $\sV$ is a cofibrantly generated monoidal model category 
(\S1.1 and \myref{monoidal}). For simplicity of exposition, we assume further that $\sV$ 
satisfies the monoid axiom \myref{monax}; as in \myref{monoid}, less stringent hypotheses suffice.  
Similarly, we might weaken the unit hypothesis as in Remarks \ref{usehow} and \ref{slippery}.

\begin{thm}\mylabel{ModCat} The  category $\sV\bO$-$\sC\!at$ is a cofibrantly generated model 
category in which a map $\al\colon \sD\rtarr \sE$ is a weak equivalence or fibration if each 
$\al\colon \sD(d,e)\rtarr \sE(d,e)$ is a weak equivalence or fibration in $\sV$; $\al$ is a 
cofibration if it satisfies the LLP with respect to the acyclic fibrations.  If $\al$ is a 
cofibration and either $\mathbf{I}$ or each $\sD(d,e)$ is cofibrant in $\sV$, then each
$\al\colon \sD(d,e)\rtarr \sE(d,e)$ is a cofibration.
\end{thm}
\begin{proof}[Sketch proof]
Define the category $\sV\bO$-$\sG\!raph$ to be the product of copies of $\sV$ indexed on the set $\bO\times \bO$. 
Thus an object is a set $\{\sC(d,e)\}$ of objects of $\sV$.  As a product of model categories,  $\sV\bO$-$\sG\!raph$
is a model category.  A map is a weak equivalence, fibration or cofibration if each of its components is so.
Say that $\sC$ is concentrated at $(d,e)$ if $\sC(d',e') = \ph$, the initial object, for $(d',e')\neq (d,e)$.   
For $V\in \sV$, write $V(d,e)$ for the graph concentrated at $(d,e)$ with value $V$ there.
The model category $\sV\bO$-$\sG\!raph$ is cofibrantly generated. Its generating cofibrations and acyclic 
cofibrations are the maps $\al(d,e)\colon V(d,e) \rtarr W(d,e)$ specified by generating cofibrations or generating acyclic cofibrations $V\rtarr W$ in $\sV$.  

The category $\sV\bO$-$\sG \! raph$ is monoidal with product denoted $\Box$. The $(d,e)^{th}$ object of $\sD \Box \sE$ 
is the coproduct over $c\in \bO$ of $\sE(c,e)\otimes \sD(d,c)$. The unit object is the $\sV\bO$-graph $\mathbf{I}$ 
with $\mathbf{I}(d,d) = \mathbf{I}$ and $\mathbf{I}(d,e) = \ph$ if $d\neq e$.  The category 
$\sV\bO$-$\sC\!at$ is the category of monoids in $\sV\bO$-$\sG\!raph$, hence there is a forgetful functor
\[ \bU\colon \sV\bO\mbox{-}\sC\!at \rtarr \sV\bO\mbox{-}\sG\!raph \] 
This functor has a left adjoint $\bF$ that constructs the free $\sV\bO$-$\sC\!at$ generated by a
$\sV\bO$-$\sG\!raph$ $\sC$. The construction is analogous to the construction of a tensor algebra.
The $\sV$-category $\bF\sC$ is the coproduct of its homogeneous parts $\bF_p\sC$ of ``degree $p$ monomials''. 
Explicitly, $\bF_0\sC = \mathbf{I}[\bO] = \amalg\ \mathbf{I}(d,d)$, $(\bF_1\sC)(d,e) = \sC(d,e)$, and, for $p>1$, 
\[ (\bF_p\sC)(d,e) = \coprod_{(d_i)} \sC(d_{p-1},e)\otimes \sC(d_{p-2},d_{p-1})
\otimes \cdots \otimes \sC(d_1,d_2)\otimes  \sC(d,d_1). \]
The unit map $\mathbf{I} \rtarr \bF(d,d)$ is given by the identity map 
$\mathbf{I}\rtarr \mathbf{I}(d,d)\subset (\bF\sC)(d,d)$.   The composition is given by the evident 
$\otimes$-juxtaposition maps. 

The generating cofibrations and acyclic cofibrations are obtained by applying $\bF$ to the 
generating cofibrations and acyclic cofibrations of $\sV\bO$-$\sG\!raph$. A standard implication
of \myref{criterion} applies to the adjunction $(\bF,\bU)$. The assumed applicability of the small object
argument to the generating cofibrations and acyclic cofibrations in $\sV$ implies its applicability to the
generating cofibrations and acyclic cofibrations in $\sV\bO$-$\sC\!at$, and condition (ii) of \myref{criterion} is a 
formal consequence of its analogue for $\sV\bO$-$\sG\!raph$.  Thus to prove the model axioms it remains only
to verify the acyclicity condition (i).  The relevent cell complexes are defined using coproducts, pushouts,
and sequential colimits in $\sV\bO$-$\sC\!at$, and the monoid axiom (or an analogous result under weaker
hypotheses) is used to prove that.  The details are essentially the same as in the one object case, which is 
treated in \cite[6.2]{SS0}, with objects $\sD(d,e)$ replacing copies of a monoid in $\sV$ in the argument. 
The proof relies on combinatorial analysis of the relevant pushouts. As noted in the proof of \cite[6.3]{SS2}, there 
is a slight caveat to account for the fact that \cite[6.2]{SS0} worked with a symmetric monoidal category,
whereas the product $\Box$ on $\sV\bO$-$\sG\!raph$ is not symmetric. However, the levelwise definition of the 
model structure on $\sV\bO$-$\sG\!raph$ allows use of the symmetry in $\sV$ at the relevant place in the proof.
\end{proof}

\section{Changing the categories $\sV$, $\sD$, and $\sM$}\label{changeDMV}

Let us return to Baez's joke and compare simplicial and topological enrichments,
among other things.  Throughout this section, we consider an adjunction
\begin{equation}\label{VW}
\xymatrix{\sV\ar@<.4ex>[r]^\bT & \sW \ar@<.4ex>[l]^\bU}
\end{equation}
between symmetric monoidal categories $\sV$ and $\sW$.  We work categorically 
until otherwise specified, ignoring model categorical structure. We also ignore 
presheaf categories for the moment.

Consider a $\sV$-category
$\sM$ and a $\sW$-category $\sN$.  Remember the distinction between thinking of
the term ``enriched category'' as a noun and thinking of ``enriched'' as an adjective
modifying ``category''.  From the former point of view, we can try to define a $\sV$-category
$\bU \sN$ by setting $\ul{\bU\sN}(X,Y) = \bU \ul{N}(X,Y)$, where $X,Y\in \sN$, and we can try to define a 
$\sW$-category $\bT \sM$ by setting $\ul{\bT\sM}(X,Y) = \bT \ul{N}(X,Y)$, where $X,Y\in \sM$.  Of
course, our attempts fail to give unit and composition laws unless the functors $\bU$ and
$\bT$ are sufficiently monoidal, but if they are then this can work in either direction.  

However, if we think of ``enriched category'' as a noun,
then we think of the underlying categories $\sM$ and $\sN$ as fixed and given.  To have our 
attempts work without changing the underlying category, we would have to have isomorphisms
\[ \sV(\mathbf{I}, \bU \ul{\sN}(X,Y)) \iso \sW(\mathbf{J},\ul{\sN}(X,Y)) \]
or 
\[ \sW(\mathbf{J}, \bT \ul{\sM}(X,Y)) \iso \sV(\mathbf{I},\ul{\sM}(X,Y)) \]
where $\mathbf{I}$ and $\mathbf{J}$ are the units of $\sV$ and $\sW$. The latter is
not plausible, but the former holds by the adjunction provided that $\bT\mathbf{I}\iso \mathbf{J}$.
We conclude that it is reasonable to transfer enrichment along a right adjoint but not along a 
left adjoint.  

In particular, if $\bT$ is geometric realization $sSet\rtarr \sU$ and $\bU$ is the
total singular complex functor, both of which are strong symmetric monoidal with respect to 
cartesian product, then $\bT\mathbf{I}\iso \mathbf{J}$ (a point) and we can pull 
back topological enrichment to simplicial enrichment without changing the underlying category, 
but not the other way around.  This justifies preferring simplicial enrichment
to topological enrichment and should allay Baez's suspicion. Nevertheless, it is sensible to 
use topological enrichment when that is what appears naturally.

\subsection{Changing the enriching category $\sV$}\label{SecChV}

We describe the categorical relationship between adjunctions and enriched categories
in more detail. The following 
result is due to Eilenberg and Kelly \cite[6.3]{EK}.  Recall that 
$\bT\colon \sV\rtarr \sW$ is lax symmetric monoidal
if we have a map $\nu\colon \mathbf{J}\rtarr \bT \mathbf{I}$ and a natural map 
$$\om\colon  \bT V\otimes \bT V'\rtarr \bT(V\otimes V')$$
that are compatible with the coherence data (unit, associativity, and symmetry
isomorphisms); $\bT$ is op-lax monoidal if the arrows point the other way, and
$\bT$ is strong symmetric monoidal if $\nu$ is an isomorphism and $\om$ is a natural 
isomorphism.  We are assuming that $\bT$ has a right adjoint $\bU$.  If $\bU$ is lax 
symmetric monoidal, then $\bT$ is op-lax symmetric monoidal via the adjoints of 
$\mathbf{I}\rtarr \bU\mathbf{J}$ and the natural composite
\[ \xymatrix@1{
 V\otimes V' \ar[r] & \bU\bT V\otimes \bU\bT V' \rtarr \bU(\bT V\otimes \bT V').} \\ \] 
The dual also holds. It follows that if $\bT$ is strong symmetric monoidal, then $\bU$ is 
lax symmetric monoidal. 

\begin{prop}\mylabel{mylem} Let $\sN$ be a bicomplete $\sW$-category. Assume that $\bU$ 
is lax symmetric monoidal and the adjoint $\bT \mathbf{I}\rtarr \mathbf{J}$ 
of the unit comparison map $\mathbf{I}\rtarr \bU \mathbf{J}$ is an isomorphism. 
Letting
\[\ul{\sM}(M,N) = \bU\ul{\sN}(M,N), \]
we obtain a $\sV$-category $\sM$ with the same underlying category as $\sN$. If, further, 
$\bT$ is strong symmetric monoidal, then $\sM$ is a bicomplete $\sV$-category\footnote{If the functor 
$\sV(\mathbf{I},-)\colon \sV\rtarr\mathrm{Set}$ is conservative (reflects isomorphisms), 
as holds for example when $\sV=\mathrm{Mod}_k$, then $\sM$ becomes a bicomplete $\sV$-category 
without the assumption that $\bT$ is strong symmetric monoidal.} with
\[M\odot V = M\odot \bT(V) \ \ \ \text{ and }\ \ \ F(V,M) = F(\bT V, M).\]
\end{prop}
\begin{proof}  Using the product comparison map 
\[ \bU \ul{\sN}(M,N) \otimes \bU \ul{\sN}(L,M)
\rtarr \bU(\ul{\sN}(M,N)\otimes \ul{\sN}(L,M)), \]
we see that the composition functors for $\ul{\sN}$ induce composition functors for 
$\ul{\sM}$.  The composites of the unit comparison map and the unit maps 
$\mathbf{J}\rtarr \ul{\sN}(M,M)$ in $\sW$ induce unit
maps $\mathbf{I}\rtarr \sM(M,M)$ in $\sV$.  As we have implicitly noted, this much makes sense 
even without the adjoint $\bT$ and would apply equally well with the roles of $\bU$ and $\bT$ reversed;
our hypotheses ensure that the underlying categories of $\sN$ and $\sM$ are the same.

Now assume that $\bT$ is strong symmetric monoidal. For each $V\in \sV$ and $W\in\sW$, 
a Yoneda argument provides an isomorphism
\[\ul{\sV}(V,\bU W)\iso \bU\ul{\sW}(\bT V,W)\]
that makes the pair of $\sV$-functors $(\bT, \bU)$ into a $\sV$-adjoint pair (\ref{Vadj}).
In particular, this gives an isomorphism 
\[\ul{\sV}(V,\bU\ul{\sN}(M,N))\iso \bU \ul{\sW}(\bT V, \ul{\sN}(M,N)). \]
By the adjunctions that define $\sW$-tensors and $\sW$-cotensors in $\sN$, this gives natural isomorphisms
\[ \bU\ul{\sN}(M\odot\bT V,N) \iso \ul{\sV}(V,\bU\ul{\sN}(M,N))
\iso \bU\ul{\sN}(M,F(\bT V, Y)) \]
which imply the claimed identification of $\sV$-tensors and $\sV$-cotensors in $\sM$.
\end{proof}
 
\begin{exmp} As observed in \myref{Good}, we have a strong monoidal functor  
$\mathbf{I}[-]\colon \mathrm{Set}\to\sV$. It is left adjoint to 
$\sV(\mathbf{I},-)\colon \sV\to\mathrm{Set}$. The change of enrichment 
given by \myref{mylem} produces the underlying category of a $\sV$-category.
\end{exmp}

\begin{exmp}\mylabel{GeomSingExmp} Consider the adjunction 
\[ \xymatrix{
sSet\ar@<.4ex>[r]^-{\bT} & \sU \ar@<0.4ex>[l]^-{\bS},  }\] 
where $\bT$ and $\bS$ are the geometric realization and total singular complex functors.    
Since $\bT$ and $\bS$ are strong symmetric monoidal, \myref{mylem} shows that any 
category enriched and bitensored over $\sU$ is canonically enriched and bitensored over $sSet$.
\end{exmp}

\begin{rem}\mylabel{VWEnrich} In (\ref{Vadj}), we consider enriched adjunctions between categories both
enriched over a fixed $\sV$.  One can ask what it should mean for the adjunction (\ref{VW})
to be enriched.  A reasonable answer is that there should be unit and counit maps 
\[ \ul{\sV}(V,V')\rtarr \ul{\sV}(\bU\bT V,\bU\bT V')  \ \ \text{and}\ \ 
\ul{\sW}(\bT\bU W,\bT\bU W')\rtarr \ul{\sW}(W,W') \]
in $\sV$ and $\sW$, respectively.  However, this fails for \myref{GeomSingExmp} since the function
\[  \sU(\bT\bS X,\bT\bS Y) \rtarr  \sU(X,Y) \]
induced by the counit is not continuous.
\end{rem}

\myref{mylem} is relevant to many contexts in which we use two related enrichments 
simultaneously. Such double enrichment is intrinsic to equivariant theory, as we 
see in \cite{GMequiv}, and to the relationship between spectra and spaces.

\begin{exmp}\mylabel{Enrichspec} Let $\sT$ be the closed symmetric monoidal category of 
nondegenerately based spaces in $\sU$ and let $\sS$ be some good closed symmetric monoidal 
category of spectra, such as the categories of symmetric or orthogonal spectra.
While interpretations vary with the choice of $\sS$, we always have a zeroth space 
(or zeroth simplicial set) functor, which we denote by $\text{ev}_0$.  It has a 
left adjoint, which we denote by $F_0$.  We might also write $F_0 = \SI^{\infty}$ and 
$\text{ev}_0 = \OM^{\infty}$, but homotopical understanding requires fibrant 
and/or cofibrant approximation, depending on the choice of $\sS$.
We assume that $F_0$ is strong symmetric monoidal, as holds for symmetric and orthogonal 
spectra \cite[1.8]{MMSS}.  By \myref{mylem}, $\sS$ is then enriched over $\sT$ as well
as over itself. The based space $\sS(X,Y)$ of maps $X\rtarr Y$ is 
\[ \sS(X,Y) = \text{ev}_0(\ul{\sS}(X,Y)). \]
\end{exmp}

Returning to model category theory, suppose that we are in the situation of \myref{mylem} 
and that $\sV$ and $\sW$ are monoidal model categories and $\sM$ is a $\sW$-model category. 
It is natural to ask under what conditions on the adjunction $(\bT,\bU)$ the resulting 
$\sV$-category $\sM$ becomes a $\sV$-model category. Recall the following definition
from \cite[4.2.16]{Hovey}.

\begin{defn}\mylabel{monQuill} A monoidal Quillen adjunction $(\bT,\bU)$ between 
symmetric mon\-oid\-al model
categies is a Quillen adjunction in which the left adjoint $\bT$ is strong symmetric 
monoidal and the map $\bT(Q \mathbf{I})\to\bT(\mathbf{I})$ is a weak equivalence. 
\end{defn}

The following result is essentially the same as \cite[A.5]{Dugger} 
(except that the compatibility of $\bT$ with a cofibrant replacement of 
$\mathbf{I}$ is not mentioned there).

\begin{prop}\mylabel{ChangeModelEnr} Let $\xymatrix{\sV\ar@<.4ex>[r]^\bT & \sW \ar@<.4ex>[l]^\bU}$
be a monoidal Quillen adjunction between symmetric monoidal model categories. 
Suppose that $\sM$ is a $\sW$-model category. Then the enrichment of $\sM$ in $\sV$ of \myref{mylem} makes $\sM$ into a $\sV$-model category. 
\end{prop}

\begin{cor} Any topological model category has a canonical structure of a simplicial model category.
\end{cor}

\subsection{Categorical changes of $\sV$ and $\sD$}

Still considering the adjunction (\ref{VW}), we now assume that $\bT$ is strong symmetric 
monoidal and therefore $\bU$ is lax symmetric monoidal.  We consider changes of presheaf 
categories in this context, working categorically in this section and model categorically 
in the next.  We need some elementary formal structure that relates categories of 
presheaves whose domain $\sV$-categories or $\sW$-categories have a common fixed object set 
$\bO = \{d\}$. To see that  the formal structure really is elementary, it is helpful to think of 
$\sV$ and $\sW$ as the categories of modules over commutative rings $R$ and $S$, and consider base 
change functors associated to a ring homomorphism $\phi\colon R\rtarr S$. To ease the translation, 
think of presheaves $\sD^{op}\rtarr \sV$ as right $\sD$-modules and covariant functors $\sD\rtarr \sV$ 
as left $\sD$-modules. This point of view was used already in \S\ref{changeDM}. We use the categories introduced in \S\ref{modelVO}.

We have two adjunctions induced by (\ref{VW}). The first is obvious, namely 
\begin{equation}\label{twoad}
 \xymatrix@1{ \sV\bO\mbox{-}\sC\!at\ \ar@<0.4ex>[r]^-{\bT}  
 &\ \sW\bO\mbox{-}\sC\!at \ar@<0.4ex>[l]^-{\bU}.\\ }
\end{equation}
The functors $\bT$ and $\bU$ are obtained by applying the functors $\bT$ and $\bU$ of (\ref{VW}) to morphism objects 
of $\sV\bO$-categories and $\sW\bO$-categories.  Since $\bT$ and $\bU$ are lax symmetric monoidal, they preserve composition.

The second is a little less obvious. Consider $\sD\in \sV\bO$-$\sC\!at$ and $\sE\in \sW\bO$-$\sC\!at$
and let $\ph\colon \sD\rtarr  \bU\sE$ be a map of $\sV$-categories;
equivalently, we could start with the adjoint $\tilde{\ph}\colon \bT\sD\rtarr \sE$.
We then have an induced adjunction
\begin{equation}\label{threead}
 \xymatrix@1{ \mathbf{Pre} (\sD,\sV)\ \ar@<0.4ex>[r]^-{\bT_{\ph}}  
 &\ \mathbf{Pre} (\sE,\sW) \ar@<0.4ex>[l]^-{\bU_{\ph}}.\\ } 
\end{equation}
To see this, let $X\in \mathbf{Pre} (\sD,\sV)$ and $Y\in \mathbf{Pre} (\sE,\sW)$.  The presheaf 
$\bU_{\ph} Y\colon \sD^{op}\rtarr \sV$ 
is defined via the adjoints of the following maps in $\sV$.
\[ \xymatrix@1{
\sD(d,e)\otimes_{\sV} \bU Y_e\ar[r]^-{\ph\otimes \id } 
& \bU\sE(d,e)\otimes_{\sV} \bU Y_e \ar[r] & \bU(\sE(d,e)\otimes_{\sW} Y_e)\ar[r] & \bU Y_d.\\} \]
The presheaf $\bT_{\ph} X\colon \sE^{op}\rtarr \sW$ is obtained by an extension of scalars that can be written 
conceptually as $\bT X\otimes_{\bT\sD}\bY$. To make sense of this, recall that we have the 
represented presheaves $\bY(e)$ such that $\bY(e)_d = \sE(d,e)$. As
$e$-varies, these define a covariant $\sW$-functor $\bY\colon \sE \rtarr \mathbf{Pre} (\sE,\sW)$. 
Pull this back via $\ph$ to obtain a covariant $\sW$-functor $\bT\sD\rtarr \mathbf{Pre} (\sE,\sW)$.
The tensor product is the coequalizer
\[  \xymatrix@1{
\coprod_{d,e}\bT X_e\otimes_{\sW} \bT \sD(d,e) \otimes_{\sW} \bY(d) \ar@<1ex>[r] \ar@<-1ex>[r] &
\coprod_d \bT X_d \otimes_{\sW} \bY(d) \ar[r] & \bT X \otimes_{\bT \sD} \bY \equiv \bT_{\ph}X,\\} \]
where the parellel arrows are given by the functors $\bT X$ and $\bY$. Composition on the right makes this a contravariant functor $\sE\rtarr \sW$. 

There are two evident special cases, which are treated in \cite[App A]{DS}.  The first is obtained by starting with $\sE$ and taking $\ph$ to be $\id\colon \bU \sE\rtarr \bU \sE$. This gives an adjunction
\begin{equation}\label{fourad}
 \xymatrix@1{ \mathbf{Pre} (\bU\sE,\sV)\ \ar@<0.4ex>[r]^-{\bT}  
 &\ \mathbf{Pre} (\sE,\sW) \ar@<0.4ex>[l]^-{\bU}.\\ } 
\end{equation}
The second is obtained by starting with $\sD$ and taking $\ph$ to be $\et\colon \sD \rtarr \bU\bT \sD$.
This gives an adjunction
\begin{equation}\label{fivead}
 \xymatrix@1{ \mathbf{Pre} (\sD,\sV)\ \ar@<0.4ex>[r]^-{\bT_{\et}}  
 &\ \mathbf{Pre} (\bT\sD,\sW) \ar@<0.4ex>[l]^-{\bU_{\et}}. \\}
\end{equation}
The adjunction (\ref{threead}) factors as the composite of the adjunction (\ref{fourad}) and an adjunction 
of the form $(\ph_!,\ph^*)$:
\begin{equation}\label{sixad}
 \xymatrix@1{ \mathbf{Pre} (\sD,\sV)\ \ar@<0.4ex>[r]^-{\ph_!} & \mathbf{Pre} (\bU\sE,\sV)\ \ar@<0.4ex>[r]^-{\bT}  
 \ar@<0.4ex>[l]^-{\ph^*}
 &\ \mathbf{Pre} (\sE,\sW) \ar@<0.4ex>[l]^-{\bU}.\\ } 
\end{equation} 
This holds since the right adjoints in (\ref{threead}) and (\ref{sixad}) are easily seen to be the same.

\subsection{Model categorical changes of $\sV$ and $\sD$}\label{VV+sec}

We want a result to the effect that if $(\bT,\bU)$ in (\ref{VW}) is a Quillen equivalence, then {for any weak equivalence $\ph\colon \sD\rtarr \bU\sE$,}
$(\bT_{\ph},\bU_{\ph})$ in (\ref{threead}) is also a Quillen equivalence.  As in \myref{subtle}, we set up a 
general context that will be encountered in the sequel \cite{GM2}; it is a variant of the context of 
\cite[\S6]{SS2}.  We assume that the identity
functor is a left Quillen equivalence $\sV_+\rtarr \sV$ for two model structures on $\sV$ with the 
same weak equivalences, where the unit 
$\mathbf{I}$ is cofibrant in $\sV$ but not necessarily in $\sV_+$.  Similarly, we assume that $\sV$ but not necessarily $\sV_+$ satisfies the monoid axiom. 
We do not assume that $\sW$ satisfies the monoid axiom, but we do assume that all presheaf categories
$\mathbf{Pre} (\sE,\sW)$ are model categories and all weak equivalences $\sE\rtarr \sE'$ in sight are $\sW$-weak equivalences in the sense of \myref{Nuisance}(iv). 

Categorically, the adjunction (\ref{VW}) is independent of model structures.  However, we assume that 
\begin{equation}\label{VWtoo}
 \xymatrix@1{ \sV_+\ar@<0.4ex>[r]^{\bT}  & \sW \ar@<0.4ex>[l]^{\bU} }.
\end{equation}
is a Quillen equivalence in which $\bU$ creates the weak equivalences in $\sV$ and that the unit 
$\et\colon V\rtarr \bU\bT V$ of the adjunction is a weak equivalence for all cofibrant $V$ in $\sV$ (not just
in $\sV_+$).  With the level model structures that we are considering, the right adjoint $\bU_{\ph}$ in the adjunction 
\begin{equation}\label{threeadtoo}
 \xymatrix@1{ \mathbf{Pre} (\sD,\sV)_+\ \ar@<0.4ex>[r]^-{\bT_{\ph}}  
 &\ \mathbf{Pre} (\sE,\sW) \ar@<0.4ex>[l]^-{\bU_{\ph}}\\ } 
\end{equation}
then creates the weak equivalences and fibrations in $\mathbf{Pre} (\sE,\sW)$, so that (\ref{threeadtoo}) is again a Quillen adjunction. With these assumptions, we have the following variant of theorems in \cite{DS,SS2}. 

\begin{thm}\mylabel{Changeflesh} If $(\bT,\bU)$ in (\ref{VWtoo}) is a Quillen equivalence and
$\ph\colon \sD\rtarr \bU\sE$ is a weak equivalence, then $(\bT_{\ph},\bU_{\ph})$ in
(\ref{threeadtoo}) is a Quillen equivalence.
\end{thm}
\begin{proof} We have a factorization of (\ref{threeadtoo}) as in (\ref{sixad}), and 
$(\ph_!,\ph^*)$ is a Quillen equivalence by \myref{flesh}. Therefore it suffices to 
consider the special case when $\ph=\mbox{id}\colon \bU\sE\rtarr \bU\sE$. 

Let $\ga\colon \bQ{\bU\sE} \rtarr \bU \sE$ 
be a cofibrant approximation in the model structure on $\sV\bO$-$\sC\!at$ of \myref{ModCat}. Since $\mathbf{I}$ is cofibrant in $\sV$, each $\bQ\bU{\sE}(d,e)$ is cofibrant and thus, by assumption, each map 
$\et\colon \bQ\bU{\sE}(d,e)\rtarr \bU\bT \bQ\bU{\sE}(d,e)$ is a weak equivalence.  Let 
$\tilde{\ga}\colon \bT\bQ\bU{\sE}\rtarr \sE$ be the adjoint of $\ga$ obtained from the 
adjunction (\ref{twoad}). Since the weak equivalence $\ga$ is the composite 
\[  \xymatrix@1{ \bQ\bU{\sE} \ar[r]^-{\et} & \bU\bT \bQ\bU{\sE} \ar[r]^-{\bU \tilde{\ga}} & \bU \sE\\} \]
and $\et$ is a weak equivalence, $\bU \tilde{\ga}$ is a weak equivalence by the two out of three
property. Since $\bU$ creates the weak equivalences, $\tilde{\ga}$ is a weak equivalence.

The identity $\bU\tilde{\ga}\com \et = \ga$ leads to a commutative square of right Quillen adjoints
\[ \xymatrix{
\mathbf{Pre} (\sE,\sW) \ar[d]_{\bU} \ar[r]^{\tilde{\ga}^*} & \mathbf{Pre} (\bT\bQ\bU{\sE},\sW) \ar[d]^{\bU_{\et}} \\
\mathbf{Pre} (\bU\sE),\sV_+) \ar[r]_-{\ga^*} & \mathbf{Pre} (\bQ\bU{\sE},\sV_+).\\} \]
By \myref{flesh}, the horizontal arrows are the right adjoints of Quillen 
equivalences. Therefore it suffices to prove that the right vertical arrow is the right adjoint of a 
Quillen equivalence.

To see this, start more generally with a cofibrant object $\sD$ in $\sV\bO$-$\sC\!at$ and consider the
Quillen adjunction 
\begin{equation}\label{fiveadToo}
 \xymatrix@1{ \mathbf{Pre} (\sD,\sV_+)\ \ar@<0.4ex>[r]^-{\bT_{\et}}  
 &\ \mathbf{Pre} (\bT \sD,\sW) \ar@<0.4ex>[l]^-{\bU_{\et}} \\}
 \end{equation}
It suffices to prove that the unit $X\rtarr \bU_{\et}\bT_{\et} X$ 
is a weak equivalence for any cofibrant $X$ in $\mathbf{Pre} (\sD,\sV_+)$. Since $X$ is also cofibrant 
in $\mathbf{Pre} (\sD,\sV)$ and each $\sD(d,e)$ is cofibrant in $\sV$, each $X_d$ is cofibrant in $\sV$ by \myref{level2}. 
Our assumption that $\et\colon V\rtarr \bU\bT V$ is a weak equivalence for all cofibrant $V$ gives the
conclusion.
\end{proof}

\subsection{Tensored adjoint pairs and changes of $\sV$, $\sD$, and $\sM$}\label{tensored}
We are interested in model categories that have approximations as presheaf categories, so we naturally want to 
consider situations where, in addition to the adjunction (\ref{VW}) between $\sV$ and $\sW$, we have a 
$\sV$-category $\sM$, a $\sW$-category $\sN$, and an adjunction
\begin{equation}\label{THREEad}
 \xymatrix@1{ \sM\ar@<0.4ex>[r]^{\bJ}  & \sN \ar@<0.4ex>[l]^{\bK}}\\ 
\end{equation}
that is suitably compatible with (\ref{VW}).  In view of our standing assumption that $\bT$ is strong
symmetric monoidal and therefore $\bU$ is lax symmetric monoidal, the following definition seems reasonable. 
It covers the situations of most interest to us, but the notion of ``adjoint module'' introduced by Dugger 
and Shipley \cite[\S\S3,4]{DS} gives the appropriate generalization in which it is only assumed that $\bU$ 
is lax symmetric monoidal. Recall the isomorphisms of (\ref{trans}).

\begin{defn}\mylabel{tenadj} The adjunction $(\bJ,\bK)$ is tensored over the adjunction $(\bT,\bU)$ if there is a
natural isomorphism 
\begin{equation}\label{adjten} \bJ X\odot \bT V \iso \bJ (X\odot V) 
\end{equation}
such that the following coherence diagrams of isomorphisms commute for $X\in\sM$ and $V,V'\in \sV$.
\[ \xymatrix{
\bJ X \ar[d] \ar[r] & \bJ (X\odot \mathbf{I}_{\sV}) \\
\bJ X \odot \mathbf{I}_{\sW}  \ar[r] &   \bJ X \odot \bT \mathbf{I}_{\sV} \ar[u].\\} \]
\[ \xymatrix{
(\bJ X\odot \bT V)\odot \bT V' \ar[r] \ar[d] & \bJ(X\odot V)\odot \bT V' \ar[r] & \bJ((X\odot V)\odot V') \ar[d]\\
\bJ X \odot (\bT V\otimes \bT V') \ar[r] & \bJ X\odot \bT(V\otimes V') \ar[r] & \bJ(X\odot (V\otimes V')).\\} \]
\end{defn}

The definition implies an enriched version of the adjunction $(\bJ,\bK)$. 

\begin{lem} If $(\bJ,\bK)$ is tensored over $(\bT,\bU)$, then there is a natural isomorphism
\[ \bU\ul{\sN}(\bJ X, Y)\iso \ul{\sM}(X,\bK Y) \]
in $\sV$, where $X\in \sM$ and $Y\in \sN$.
\end{lem}
\begin{proof} For $V\in \sV$, we have the sequence of natural isomorphisms
\begin{eqnarray*}
\sV(V,\bU\ul{\sN}(\bJ X,Y) & \iso & \sW(\bT V,\ul{\sN}(\bJ X,Y)) \\
&\iso & \sN(\bJ X\odot \bT V,Y)\\
&\iso & \sN(\bJ(X\odot V), Y) \\
& \iso & \sM(X \odot V, \bK Y)\\
& \iso & \sV(V, \ul{\sM}(X,\bK Y). 
\end{eqnarray*}
The conclusion follows from the Yoneda lemma.
\end{proof}

We are interested in comparing presheaf categories $\mathbf{Pre} (\sD,\sV)$ and $\mathbf{Pre} (\sE,\sW)$ where 
$\sD$ and $\sE$ are full categories of bifibrant objects that correspond under a Quillen 
equivalence between $\sM$ and $\sN$. In the context of \S\ref{VV+sec}, we can change
$\sV$ to $\sV_+$. The following results then combine with \myref{subtle} and \myref{Changeflesh} 
to give such a comparison. 

\begin{thm}\mylabel{tensorthm}  Let $(\bJ,\bK)$ be tensored over $(\bT,\bU)$, where $(\bJ,\bK)$ is a Quillen 
equivalence. Let $\sE$ be a small full $\sW$-subcategory of bifibrant objects of $\sN$. Then 
$\bU\sE$ is quasi-equivalent to the small full $\sV$-subcategory $\sD$ of $\sM$ with bifibrant
objects the $\bQ\bK Y$ for $Y\in\sE$, where $\bQ$ is a cofibrant approximation functor
in $\sM$. 
\end{thm}
\begin{proof}  We define a $(\bU\sE,\sD)$-bimodule $\sF$.  Let $X,Y,Z\in \sE$. Define 
\[ \sF(X,Y) = \sM(\bQ\bK X,\bK Y).\]  
The right action of $\sD$ is given by composition
\[  \sM(\bQ\bK Y,\bK Z) \otimes \sM(\bQ\bK X,\bQ\bK Y) \rtarr \sM(\bQ\bK X,\bK Z). \]
The counit $\bJ\bK\rtarr \Id$ of the adjunction gives a natural map 
\[  \bU\ul{\sN}(X,Y)\rtarr \bU\ul{\sN}(\bJ\bK X,Y)\iso \ul{\sM}(\bK X,\bK Y). \]
The left action of $\bU\sE$ is given by the composite
\[ \bU\ul{\sN}(Y,Z)\otimes \ul{\sM}(\bQ\bK X, \bK Y) \rtarr  \ul{\sM}(\bK Y, \bK Z)\otimes \ul{\sM}(\bQ\bK X, \bK Y)
\rtarr \ul{\sM}(\bQ\bK X, \bK Z). \]
Using the coherence diagrams in \myref{tenadj}, a lengthy but routine check shows that the diagrams that are required
to commute in \S \ref{QEchangesec} do in fact commute. Define $\ze_X\colon \bf{I}\rtarr \sF(X,X)$
to be the composite
\[ \bf{I} \rtarr \ul{\sM}(\bQ\bK X, \bQ\bK X) \rtarr \ul{\sM}(\bQ\bK X, \bK X) \]
induced by the weak equivalence $\bQ\bK X\rtarr \bK X$. 
By the naturality square
\[\xymatrix{
 \bU \ul{\sN}(\bJ\bK X,Y)\ar[r]^{\iso}  \ar[d] & \ul{\sM}(\bK X, \bK Y) \ar[d] \\  
 \bU \ul{\sN}(\bJ\bQ\bK X, Y) \ar[r]_{\iso}    & \ul{\sM}(\bQ\bK X, \bK Y) \\} \]
the map
\[ (\ze_X)^*\colon \bU \ul{\sN}(X,Y)\rtarr \ul{\sM}(\bQ\bK X, \bK Y)  \]
is the composite
\[ \bU \ul{\sN}(X,Y) \rtarr \bU \ul{\sN}(\bJ\bK X,Y) \rtarr \bU\ul{\sN}(\bJ\bQ\bK X,Y) 
\iso \ul{\sM}(\bQ\bK X, \bK Y). \]
Since $(\bJ,\bK)$ is a Quillen equivalence, the composite $\bJ\bQ\bK X\rtarr \bJ\bK X\rtarr X$
is a weak equivalence, hence $(\ze_X)^*$ is a weak equivalence by \myref{Cutesy}.
The map 
\[ (\ze_Y)_*\colon \ul{\sM}(\bQ\bK X,\bQ\bK Y) \rtarr \ul{\sM}(\bQ\bK X,\bK Y) \]
is also a weak equivalence by \myref{Cutesy}. 
\end{proof}

\begin{cor}\mylabel{pickypick} With the hypotheses of \myref{tensorthm}, let $\sD$ be a small full
$\sV$-subcategory of bifibrant objects of $\sM$. Then $\sD$ is quasi-equivalent 
to $\bU\sE$, where $\sE$ is the small full $\sW$-subcategory of $\sN$ with bifibrant 
objects the $\bR\bJ X$ for $X\in \sD$, where $\bR$ is a fibrant approximation functor
in $\sN$.
\end{cor}
\begin{proof}  By \myref{tensorthm}, $\bU\sE$ is quasi-equivalent to $\sD'$, where $\sD'$
is the full $\sV$-subcategory of $\sM$ with objects the $\bQ\bK\bR\bJ X$, and of 
course  $\bQ\bK\bR\bJ X$ is weakly equivalent to $X$. The conclusion follows from
\myref{CutesyTwee}.
\end{proof}

\subsection{Weakly unital $\sV$-categories and presheaves}\label{weakly}

In the sequel \cite{GM2}, we shall encounter a topologically motivated variant 
of presheaf categories.  Despite the results of the previous section, which show how to 
compare full enriched subcategories $\sD$ of categories $\sM$ with differing enriching categories
$\sV$, 
the choice of $\sV$ can significantly affect the mathematics when seeking simplified equivalents of full subcategories of $\sV$-categories $\sM$.  We often want the objects of $\sD$
to be cofibrant but, when $\mathbf{I}$ is not cofibrant in $\sV$, that desideratum can conflict with the requirement
that $\sD$ have strict units given by maps $\mathbf{I}\rtarr \sD(d,d)$ in $\sV$.  We shall encounter domains $\sD$ for
presheaf categories in which $\sD$ is not quite a category since a chosen cofibrant approximation
$\bQ\mathbf{I}$ rather than $\mathbf{I}$ itself demands to be treated as if it were a unit object. 
The examples start with a given fixed $\sM$ but are not full $\sV$-subcategories of $\sM$.  Retaining our 
standing assumptions on $\sV$, we conceptualize the situation with the following definitions.  We 
fix a weak equivalence $\ga\colon \bQ\mathbf{I}\rtarr \mathbf{I}$, not necessarily a fibration.  

\begin{defn} Fix a $\sV$-model category $\sM$ and a set $\bO = \{d\}$ of objects of $\sM$.  
A weakly unital $\sV$-category $\sD$ with object set $\bO$ consists of objects $\sD(d,e)$ of 
$\sV$ for $d,e\in \bO$, an associative pairing $\sD(d,e)\otimes \sD(c,d)\rtarr \sD(c,e)$, and, 
for each $d\in \bO$, a map $\et_d\colon \bQ\mathbf{I}\rtarr \sD(d,d)$ and a weak equivalence 
$\xi_d\colon d\rtarr d$ in $\sM$ that induces weak equivalences 
$$\xi_d^*\colon \sD(d,e) \rtarr \sD(d,e) \ \ \ \text{and} \ \ \ {\xi_d}_*\colon \sD(c,d)\rtarr \sD(c,d) $$
for all $c,e\in \bO$. The following unit diagrams must commute.
\[ \xymatrix{
\sD(d,e)\otimes \bQ\mathbf{I} \ar[r]^-{\mbox{id}\otimes \et_d} \ar[d]_{\xi_d^*\otimes \ga} 
& \sD(d,e)\otimes \sD(d,d)  \ar[d]^{\com}\\
\sD(d,e)\otimes \mathbf{I} \ar[r]_-{\iso} & \sD(d,e).\\}
\ \ \mbox{and} \ \ 
\xymatrix{
\bQ\mathbf{I}\otimes \sD(c,d)\ar[r]^-{\et_d\otimes\mbox{id}} \ar[d]_{\ga\otimes {\xi_d}_*} &
\sD(d,d)\otimes \sD(c,d) \ar[d]^{\com}\\
\mathbf{I}\otimes \sD(c,d) \ar[r]_{\iso} & \sD(c,d).\\} \]
A weakly unital $\sD$-presheaf is a $\sV$-functor $X\colon \sD^{\text{op}}\rtarr \sV$ defined as usual, 
except that the unital property requires commutativity of the following diagrams for $d\in \bO$.
\[ \xymatrix{
\bQ\mathbf{I} \ar[r]^-{\et_d} \ar[d]_{\ga}& \sD(d,d) \ar[d]^-{X}  \\
\mathbf{I} \ar[r]_-{\xi_d^*} & \ul{\sV}(X_d,X_d).\\} \]
Here the bottom arrow is adjoint to the map  $X(\xi_d)\colon X_d\rtarr X_d$.
We write $\mathbf{Pre} (\sD,\sV)$ for the category of weakly unital presheaves. The morphisms are the $\sV$-natural
transformations, the definition of which requires no change.
\end{defn}

\begin{rem} A $\sV$-category $\sD$ may be viewed as a weakly unital $\sV$-category $\sD'$ 
by taking $\et_d = \et\com\ga$, where $\et\colon \mathbf{I}\rtarr \sD(d,d)$
is the given unit, and taking $\xi_d = \id$. Then any $\sD$-presheaf can be viewed as a 
$\sD{'}$-presheaf.  In principle, $\sD{'}$-presheaves are slightly more general, since it is possible
for the last diagram to commute even though the composites 
\[ \xymatrix@1{ \mathbf{I}\ar[r]^-{\et} & \sD(d,d) \ar[r]^-{X} & \ul\sV(X_d,X_d)\\} \]
are not the canonical unit maps $\et$. However, this cannot happen if $\ga$ is an epimorphism in $\sV$,
in which case the categories $\mathbf{Pre} (\sD,\sV)$ and $\mathbf{Pre} (\sD',\sV)$ are identical.
\end{rem}

Virtually everything that we have proven when $\mathbf{I}$ is not cofibrant applies with minor changes 
to weakly unital presheaf categories.

\section{Appendix: Enriched model categories}\label{enriched}

\subsection{Remarks on enriched categories}\label{EnHyp}

The assumption that the symmetric monoidal category $\sV$ is closed ensures that we have an adjunction
\begin{equation}\label{inthom}
\sV(V \otimes W, Z) \iso \sV(V,\ul{\sV}(W,Z))
\end{equation}
of set-valued functors and also a $\sV$-adjunction
\begin{equation}\label{inthomtoo}
\ul{\sV}(V \otimes W, Z) \iso \ul{\sV}(V,\ul{\sV}(W,Z))
\end{equation}
of $\sV$-valued functors. 

In general, a $\sV$-adjunction 
\[ \xymatrix@1{ \sN\ar@<0.4ex>[r]^{\bT}  & \sM \ar@<0.4ex>[l]^{\bU} }\\ \]
between $\sV$-functors 
$\bT$ and $\bU$ is given by a binatural isomorphism 
\begin{equation}\label{Vadj}
\ul{\sM}(\bT N, M)\iso \ul{\sN}(N,\bU M)
\end{equation}
in $\sV$.  Applying $\sV(\mathbf{I},-)$, it induces an adjunction 
$\sM(\bT N, M)\iso \sN(N,\bU M)$ on underlying categories.
One characterization is that a $\sV$-functor $\bT$ has a right
$\sV$-adjoint if and only if $\bT$ preserves tensors (see below) and its 
underlying functor has a right adjoint in the usual set-based
sense \cite[II.6.7.6]{BorII}.  The dual characterization holds for
the existence of a left adjoint to $\bU$. We gave a generalization of the notion
of an enriched adjunction that allows for a change of $\sV$ in \S\ref{tensored}.

The assumption that $\sM$ is bicomplete means that $\sM$ has all weighted limits and
colimits \cite{Kelly}.  Equivalently, $\sM$ is bicomplete in the usual
set-based sense, and $\sM$ has tensors $M\odot V$ and cotensors $F(V,M)$. 

\begin{rem} These notations are not standard.
The standard notation for $\odot$ is $\otimes$, with obvious ambiguity. 
The usual notation for $F(V,M)$ is $[V,M]$ or $M^V$, neither of which seems 
entirely standard or entirely satisfactory.
\end{rem}

The $\sV$-product $\otimes$ between $\sV$-categories $\sM$ and $\sN$
has objects the pairs of objects $(M,N)$ and has hom objects in $\sV$ 
\[ \ul{\sM\otimes \sN}((M,N),(M',N')) = \ul{\sM}(M,M')\otimes \ul{\sN}(N,N'), \]
with units and composition induced in the evident way from those of $\sM$ and $\sN$. 
By definition, tensors and cotensors are given by $\sV$-bifunctors
\[ \odot\colon \sM\otimes \sV\rtarr \sM \ \ \ \text{and} \ \ \  
F\colon \sV^{\text{op}}\otimes \sM\rtarr \sM\]
that take part in $\sV$-adjunctions
\begin{equation}\label{biten} 
\ul{\sM}(M\odot V, N)\iso \ul{\sV}(V,\ul{\sM}(M,N))\iso \ul{\sM}(M,F(V,N)).
\end{equation}

We often write tensors as $V\odot M$ instead of $M\odot V$.  In principle, 
since tensors are defined  by a universal property and are therefore only defined 
up to isomorphism, there is no logical preference. However, in practice, we usually 
have explicit canonical constructions which differ 
by an interchange isomorphism. When $\sM=\sV$, we have the tensors and cotensors
\[ V\odot W = V\otimes W \ \ \ \text{and}\ \ \ F(V,W)=\ul{\sV}(V,W).  \]  

While (\ref{biten}) is the correct categorical definition \cite{BorII, Kelly}, one 
sometimes sees the definition given in the unenriched sense of ordinary adjunctions
\begin{equation}\label{biten2} 
{\sM}(M\odot V, N)\iso {\sV}(V,\ul{\sM}(M,N))\iso {\sM}(M,F(V,N)).
\end{equation}
These follow by applying the functor $\sV(I,-)$
to the adjunctions in (\ref{biten}).  There is a partial converse to this implication.
It is surely known, but we have not seen it in the literature.

\begin{lem} Assume that we have the first of the ordinary adjunctions (\ref{biten2}).
Then we have the first of the enriched adjunctions (\ref{biten}) if and only if we
have a natural isomorphism
\begin{equation}\label{trans}
(M\odot V)\odot W \iso M\odot(V\otimes W).
\end{equation}
Dually, assume that we have the second of the ordinary adjunctions (\ref{biten2}).
Then we have the second of the enriched adjunctions (\ref{biten}) if and only if we
have a natural isomorphism
\begin{equation}\label{trans2}
F(V,F(W,M)) \iso F(V\otimes W,M). 
\end{equation}
\end{lem}
\begin{proof} For objects $N$ of $\sM$, we have natural isomorphisms
\[  \sM((M\odot V)\odot W, N) \iso \sV(W,\ul{\sM}(M\odot V, N)) \] 
and
\[ \sM(M\odot (V\otimes W),N) \iso \sV(V\otimes W,\ul{\sM}(M,N))\iso \sV(W,\ul{\sV}(V,\ul{\sM}(M,N))). \]
The first statement follows from the Yoneda lemma.  The proof of the second statement is dual.
\end{proof}

Since we take (\ref{biten}) as a standing assumption, we have the isomorphisms (\ref{biten2}), 
(\ref{trans}), (\ref{trans2}).  We have used some other standard maps and isomorphisms without 
comment.  In particular, there is a natural map, sometimes an isomorphism,
\begin{equation}\label{keymap}
\om\colon \ul{\sM}(M,N) \otimes V\rtarr \ul{\sM}(M,N\odot V).
\end{equation}
This map in $\sV$ is adjoint to the map in $\sM$ given by the evident evaluation map
\[ M\odot (\ul{\sM}(M,N)\otimes V) \iso (M\odot \ul{\sM}(M,N))\odot V 
\rtarr N\odot V.\] 

\begin{rem}\mylabel{underly} In the categorical literature, it is standard to let $\sM_0$ 
denote the underlying category of an enriched category $\sM$. Then $\sM_0(M,N)$ denotes 
a morphism set of $\sM_0$ and $\sM(M,N)$ denotes a hom object in $\sV$.  
This notation is logical, but its conflict
with standard practice in the rest of mathematics is obtrusive We
therefore use notation closer to that of the topological and model 
categorical literature. 
\end{rem}

\subsection{Remarks on cofibrantly generated model categories}\label{cofgensec}

\begin{rem}\mylabel{cofgen} Although we have used the standard phrase 
``cofibrantly generated'', we more often have in mind ``compactly generated'' 
model categories. Compact generation, when applicable, allows one to use ordinary 
sequential cell complexes, without recourse to distracting transfinite considerations.
The cell objects are then very much closer to the applications and intuitions 
than are the transfinite cell objects that are standard in the model category 
literature.  Full details of this variant are in \cite{MP}; see also \cite{MS}.  
The point is that the standard enriching categories $\sV$ are compactly 
generated, and so are their associated presheaf categories $\mathbf{Pre} (\sD,\sV)$.
Examples of compactly generated $\sV$ include simplicial sets,
topological spaces, spectra (symmetric, orthogonal, or $S$-modules), and chain 
complexes over commutative rings.
\end{rem}

We sometimes write $\mathcal{I}_{\sM}$ and $\mathcal{J}_{\sM}$ for given sets of generators 
for the cofibrations and acyclic cofibrations of a cofibrantly generated model category $\sM$.
We delete the subscript when $\sM=\sV$. We recall one of the variants of the standard 
characterization of such model categories
(\cite[11.3.1]{Hirsch}, \cite[15.2.3]{MP}, \cite[4.5.6]{MS}). The latter two sources
include details of the compactly generated variant.  As said before, we assume familiarity with the small object argument.
It applies to the construction of both compactly and cofibrantly generated
model categories, more simply for the former. 

Recall that, for
a set of maps $\mathcal{I}$, a relative $\mathcal{I}$-cell complex is a map $A\rtarr X$ 
such that $X$ is a possibly transfinite colimit of objects $X_i$ starting with 
$X_0 = A$. For a limit ordinal $\be$, $X_{\be} = \colim_{\al<\be}X_{\al}$.
For a successor ordinal $\al+1$, $X_{\al+1}$ is the pushout
of a coproduct (of restricted size) of maps in $\mathcal{I}$ along a map from the domain of the coproduct  
into $X_{\al}$.  (Some standard sources reindex so that only one cell is attached at each stage, 
but there is no mathematical point in doing so and in fact that loses naturality; see \cite{MP}.)
In the compact variant, we place no restrictions on the cardinality
of the coproducts and only use countable sequences $\{X_i\}$. 

\begin{defn}\mylabel{catweak} 
A subcategory $\sW$ of a category $\sM$ is a category of weak equivalences if it 
contains all isomorphisms, is closed under retracts, and satisfies the two
out of three property.  A set $\mathcal{J}$ of maps in $\sW$ satisfies the 
acyclicity condition if every relative $\mathcal{J}$-cell complex $A\rtarr X$ is in $\sW$.
\end{defn} 

\begin{rem}  The acyclicity condition captures the crucial point 
in proving the model axioms. In practice, since coproducts
and sequential colimits generally preserve weak equivalences, the proof 
that a given set $\mathcal{J}$ satisfies it boils down to showing that a pushout
of a map in $\mathcal{J}$ is in $\sW$. The verification may be technically 
different in different contexts.  In topological situations, a general discussion 
and precise axiomatizations of how this property can be verified are given in 
\cite[4.5.8, 5.46]{MS}, which apply to all topological situations
the authors have encountered.  
\end{rem}

Write $\mathbf{K}^{\boxslash}$ for the class of maps in $\sM$ that satisfy
the right lifting property (RLP) with respect to a class of maps $\mathbf{K}$. 
Dually, write $^{\boxslash}\mathbf{K}$ for the class of maps in $\sM$ that satisfy
the left lifting property (LLP) with respect to $\mathbf{K}$. 

\begin{rem}  Let $(\sW,\sC,\sF)$ be a model structure on $\sM$. Then 
\[ \sF = (\sW\cap\sC)^{\boxslash}  \ \ \ \text{and} \ \ \ \sW\cap \sF = \sC^{\boxslash} \]
or equivalently
\[ \sC =\,  ^{\boxslash}(\sW\cap \sF) \ \ \ \text{and} \ \ \ \sW\cap\sC =\, ^{\boxslash}\sF. \]
Therefore $\sC$ and $\sC\cap \sW$ must be saturated, that is, closed under pushouts, 
transfinite colimits, and retracts. In particular, any subset $\mathcal{J}$ of $\sW$ satisfies the acyclicity condition.  
No matter how one proves the model axioms, getting at the saturation of 
$\sW\cap\sC$ is the essential point. Identifying a convenient subset of $\sW$ satisfying the acyclicity
condition often works most simply.
\end{rem}

\begin{thm}\mylabel{criterion} Let $\sW$ be a subcategory of weak 
equivalences in a bicomplete category $\sM$ and let $\mathcal{I}$ and $\mathcal{J}$ 
be sets of maps which permit the small object argument. Then $\sM$ is a cofibrantly
generated model category with generating cofibrations $\mathcal{I}$ and 
generating acyclic cofibrations $\mathcal{J}$ if and only if the following 
two conditions hold.
\begin{enumerate}[(i)]
\item $\sJ$ satisfies the acyclicity condition.
\item $\mathcal{I}^{\boxslash} = \sW\cap \mathcal{J}^{\boxslash}$.
\end{enumerate}
\end{thm}

In words (ii) says that a map has the RLP with respect to $\mathcal{I}$ if and only
if it is in $\sW$ and has the RLP with respect to $\mathcal{J}$.  It 
leads to the conclusion that $ \sC^{\boxslash} = \sW \cap \sF$. Its
verification is often formal, as in the following remark. 

\begin{rem}  The generating cofibrations and acyclic cofibrations $\mathcal{I}$ and 
$\mathcal{J}$ of the enriching categories $\sV$ that we are interested in 
satisfy (i) and (ii). We construct new model categories by applying a left adjoint 
$F\colon \sV\rtarr \sM$ to obtain generating sets $F\mathcal{I}$ and $F\mathcal{J}$
in $\sM$.  Then condition (ii) is inherited by adjunction from $\sV$, and the adjunction 
reduces the small object argument hypothesis to a smallness condition in $\sV$ 
that is usually easy to verify.  Therefore only (i) needs proof.
\end{rem}

\begin{rem} There are many variants of \myref{criterion}.  In some recent work, the
cofibrantly generated model category $\sM$ is assumed to be locally presentable, 
and then $\sM$ is said to be a combinatorial model category.  This 
ensures that there are no set theoretical issues with the small object argument,
and it allows alternative recognition criteria in which $\mathcal{J}$ is not given 
a priori; see for example \cite[\S A.2.6]{Lurie}.  However, in one variant 
\cite[A.2.6.8]{Lurie}, it is assumed a priori that classes $\sC$ and $\sW$ are
given such that $\sW\cap\sC$ is saturated. In another \cite[A.2.6.13]{Lurie}, 
conditions are formulated that imply directly that $\sW$ is closed under 
pushouts by cofibrations.  It seems unlikely to us that the conditions
required of $\sW$ (especially \cite[A.2.6.10(4)]{Lurie}) hold in the examples 
we are most interested in. 
\end{rem}

\subsection{Remarks on enriched model categories} 

Let $\sM$ be a model category and a $\sV$-category.  The weak equivalences, 
fibrations, and cofibrations live in the underlying category of $\sM$.  

\begin{defn}\mylabel{VModCat} We say that 
$\sM$ is a $\sV$-model category if the map
\begin{equation}\label{tencoten1} 
\ul{\sM}(i^*,p_*)\colon \ul{\sM}(X,E)\rtarr \ul{\sM}(A,E)
\times_{\ul{\sM}(A,B)} \ul{\sM}(X,B) 
\end{equation}
induced by a cofibration $i\colon A\rtarr X$ and fibration 
$p\colon E\rtarr B$ in $\sM$ is a fibration in $\sV$ which is a weak equivalence if either $i$ or $p$
is a weak equivalence. 
\end{defn}
 
The relationship of
(\ref{tencoten1}) with lifting properties should be clear.
By adjunction, as in \cite[4.2.2]{Hovey} or \cite[16.4.5]{MP}, the following two conditions are each 
equivalent to the properties required of (\ref{tencoten1}).
First, for a cofibration $i\colon A\rtarr X$ in $\sV$ and a cofibration
$j\colon B\rtarr Y$ in $\sM$, the pushout product
\begin{equation}\label{tencoten2} 
i\Box j\colon A\odot Y\cup_{A\odot B} X\odot B\rtarr X\odot Y
\end{equation}
is a cofibration in $\sM$ which is a weak equivalence if either $i$ or 
$j$ is a weak equivalence.  Second, for a cofibration $i\colon A\rtarr X$ 
in $\sV$ and a fibration $p\colon E\rtarr B$ in $\sM$, the induced map
\begin{equation}\label{tencoten3} 
F(i^*,p_*)\colon F(X,E)\rtarr F(A,E)\times_{F(A,B)} F(X,B) 
\end{equation}
is a fibration in $\sM$ which is a weak equivalence if either $i$ or 
$p$ is a weak equivalence. 

\begin{defn}\mylabel{monoidal} The model structure on $\sV$ is said to be 
monoidal if the equivalent conditions (\ref{tencoten1}), (\ref{tencoten2}),
and (\ref{tencoten3}) hold when $\sM=\sV$ and if the unit axiom holds:
if $q\colon Q\bI \rtarr I$ is a cofibrant approximation, then the 
map $V\otimes \bQ\mathbf{I}\rtarr V\otimes \mathbf{I}\iso V$ induced by $q$ is a weak equivalence 
for all cofibrant objects $V\in \sV$. When $\sV$ is monoidal, we say that a 
$\sV$-category $\sM$ is a $\sV$-model category if (\ref{tencoten1}), (\ref{tencoten2}),
and (\ref{tencoten3})
hold and the map  $M\odot \bQ\mathbf{I}\rtarr M\odot \mathbf{I}\iso M$ 
induced by a cofibrant approximation $q\colon \bQ\mathbf{I}\rtarr \mathbf{I}$ is a weak equivalence for all 
cofibrant objects $M\in\sM$.
\end{defn}

\begin{rem}\label{UnitAxHypothesis} By a cofibrant approximation, we mean a weak
equivalence, not necessarily a fibration,  with cofibrant domain.  If the unit axiom holds 
for any one given cofibrant approximation $q\colon Q\bI \rtarr I$, then it holds for all 
cofibrant approximations.  
\end{rem}

When $\sM$ is a $\sV$-model category, the homotopy category $\text{Ho}\sM$ is enriched 
over $\text{Ho}\sV$, with 
$\ul{Ho\sM}(M,N)$ represented by the object $\ul{\sM}(M,N)$ of 
$\sV$ when $M$ and $N$ are bifibrant (cofibrant and fibrant). 
We write $[M,N]_{\sM}$ for the
hom sets in $\text{Ho}\sM$, and similarly for $\sV$.  We then have 
\begin{equation}\label{Hohomset}
[M,N]_{\sM} = [\mathbf{I},\ul{Ho\sM}(M,N)]_{\sV}.
\end{equation}
The additional unit assumptions of \myref{monoidal} are necessary for the proof.  A
thorough exposition using the notion of a semicofibrant object to weaken hypotheses
is given by Lewis and Mandell \cite{LM}.  Their paper gives a comprehensive 
treatment of passage from enriched model categories to their homotopy categories in a  context closely related to ours, but with different emphases. 
It focuses on module categories over a commutative monoid $A$ in $\sV$.  If we regard $A$ as a $\sV$-category with a single object, then the category
of right $A$-modules can be identified with $\mathbf{Pre}(A,\sV)$, and many of their results generalize to our setting.

It is worth emphasizing what we have {\em not} required of $\sV$ and $\sM$. We have
not required that the unit of $\sV$ be cofibrant.  That holds in some but 
not all of the most commonly used enriching categories.  It is important
to know when it is needed and when not, and we have been careful to show the places 
where it comes into play.  We return to this point in \S\ref{How?}.   We have also not required $\sV$ to satisfy the 
monoid axiom. We recall it and formulate its analogue for $\sV$-categories, which we have
also not required.

\begin{defn}\mylabel{monax} $\sV$ satisfies the monoid axiom 
if all maps in $\sV$ that are obtained as pushouts or filtered colimits of maps of the form
$\id\otimes i\colon U\otimes V\rtarr U\otimes W$, where $i\colon V\rtarr W$ is an acyclic
cofibration in $\sV$, are weak equivalences.
\end{defn}

\begin{defn}\mylabel{tenax} Let $\sM$ be a $\sV$-model category and $\sD$ be a small
$\sV$-category.  Then $\sM$  satisfies the tensor axiom if the following conditions 
(i) and (ii) hold, and $\sM$ satisfies the $\sD$-tensor axiom if (iii) holds.
\begin{enumerate}[(i)]
\item  All maps in $\sM$ that are obtained as pushouts or filtered colimits of maps of the
form $\id\odot i\colon M\odot V\rtarr M\odot W$, where $i\colon V\rtarr W$ is an acyclic
cofibration in $\sV$, are weak equivalences.
\item All maps in $\sM$ that are obtained as pushouts or filtered colimits of maps of the
form $i\odot \id \colon M\odot V\rtarr N\odot V$, where $i\colon M\rtarr N$ is an acyclic
cofibration in $\sM$, are weak equivalences.
\item  All maps in $\sM$ that are obtained as pushouts or filtered colimits of maps of the
form $i\odot \id \colon M\odot \sD(d,e)\rtarr N\odot \sD(d,e)$, where $i\colon M\rtarr N$ is an acyclic
cofibration in $\sM$, are weak equivalences.
\end{enumerate} 
\end{defn}

\begin{rem} As observed in \cite[3.4]{SS0}, the monoid axiom holds if all objects
of $\sV$ are cofibrant. However, it often holds even when that fails. By the same argument, 
part (i) of the tensor axiom holds if all objects of $\sM$ are cofibrant and part (ii)
holds if all objects of $\sV$ are cofibrant. Restricting to $\sD$, (iii) holds if all
$\sD(d,e)$ are cofibrant.  This gives an advantage to enriching in simplicial sets,
but again, these conditions often hold when not all objects are cofibrant.
\end{rem}

We recalled the characterization of enriched adjunctions in \S\ref{EnHyp}.  
Model categorically, we are interested in Quillen $\sV$-adjunctions.

\begin{defn}\mylabel{Quillad} A Quillen $\sV$-adjunction is a $\sV$-adjunction
such that the induced adjunction on underlying model categories 
is a Quillen adjunction in the usual sense.  (In 
\cite[4.2.18]{Hovey}, the left adjoint $\bT$ is then called a $\sV$-Quillen functor.) 
A Quillen $\sV$-equivalence is a Quillen $\sV$-adjunction 
such that the induced adjunction on underlying model categories 
is a Quillen equivalence in the usual sense.
\end{defn}

\subsection{The level model structure on presheaf categories}\label{ModelD}

For a small $\sV$-category $\sD$ and any $\sV$-category $\sM$, 
$\mathbf{Fun} (\sD^{op},\sM)$ denotes the category of $\sV$-functors
$\sD\rtarr \sM$ and $\sV$-natural transformations.  When
$\sM=\sV$ we use the alternative notation $\mathbf{Pre} (\sD,\sV) = \mathbf{Fun} (\sD^{op},\sV)$.
Good references for the general structure of such categories are \cite{BorII, DS, SS2, Shul},
and we shall say more in \S\ref{CatPre}.  
We write $\mathbf{Fun} (\sD^{op},\sM)(X,Y)$ and $\mathbf{Pre} (\sD,\sV)(X,Y)$ for morphism sets in these
categories.
  
We remind the reader that we have no interest in the underlying category of $\sD$, 
and we write
$\sD(d,e)$ rather than $\ul{\sD}(d,e)$ for its hom objects in $\sV$.
It is standard, especially in additive situations, to think of a small
$\sV$-category $\sD$ as a kind of categorical ``ring with many objects''
and to think of (contravariant) $\sV$-functors defined on $\sD$ as 
(right) $\sD$-modules. Many ideas and proofs become more transparent when 
first translated to the language of rings and modules.

Let $X$ be an object of $\mathbf{Fun} (\sD^{op},\sM)$.  Writing $d \mapsto X_d$,
$X$ is given by maps
$$  X(d,e)\colon {\sD}(d,e)\rtarr \ul{\sM}(X_e,X_d)$$ 
in $\sV$. The $\sV$-natural transformations $f\colon X\rtarr Y$ are given by 
maps $f_d\colon X_d\rtarr Y_d$ in $\sM$ such that the following diagrams commute 
in $\sV$.
\begin{equation}\label{Vnat}
\xymatrix{
{\sD}(d,e) \ar[r]^-{X} \ar[d]_-{Y} & \ul{\sM}(X_e,X_d) \ar[d]^{(f_d)_{*}}\\
\ul{\sM}(Y_e,Y_d) \ar[r]_-{(f_e)^{*}} & \ul{\sM}(X_e,Y_d).\\}
\end{equation}

The category $\mathbf{Fun} (\sD^{op},\sM)$ is the underlying category of a 
$\sV$-category, as we shall explain in \S\ref{CatPre}, where we say more
about the relevant enriched category theory. We focus here on
the model categories of presheaves relevant to this paper.

\begin{defn}\mylabel{level}  Let $\sM$ be a cofibrantly generated $\sV$-model category. A map 
$f\colon X\rtarr Y$ in $\mathbf{Fun} (\sD^{op},\sM)$ is a level weak equivalence or level fibration if 
each map $f_d\colon X_d\rtarr Y_d$ is a weak equivalence or fibration in $\sM$. Recall the functors 
$F_d\colon \sM\rtarr \mathbf{Fun} (\sD^{op},\sM)$ from \myref{Fd} and define 
$F\mathcal{I}_{\sM}$ and $F\mathcal{J}_{\sM}$ 
to be the sets of all maps $F_di$ and $F_dj$ in $\mathbf{Fun} (\sD^{op},\sM)$, where $d\in \sD$, 
$i\in \mathcal{I}_{\sM}$, and $j\in \mathcal{J}_{\sM}$. 
\end{defn} 

\begin{thm}\mylabel{level2} If $F\mathcal{I}_{\sM}$ and $F\mathcal{J}_{\sM}$ admit the small object argument and $F\mathcal{J}_{\sM}$ satisfies the acyclicity condition, then $\mathbf{Fun} (\sD^{op},\sM)$ is a cofibrantly generated $\sV$-model category under the level weak equivalences and level fibrations; the sets $F\mathcal{I}_{\sM}$ and $F\mathcal{J}_{\sM}$ are the generating cofibrations and generating acyclic cofibrations. If $\sM$ is proper, then so is $\mathbf{Fun} (\sD^{op},\sM)$. The adjunctions 
$(F_d,\text{ev}_d)$ are Quillen adjunctions. If the functors ${\sD}(e,d)\odot (-)$ preserve cofibrations, then cofibrations in $\mathbf{Fun} (\sD^{op},\sM)$ are level cofibrations, hence cofibrant objects are level cofibrant. 
\end{thm}
\begin{proof}  We have assumed the smallness condition and condition (i)
of \myref{criterion}, and condition (ii) is inherited by adjunction from 
$\sM$. Since pushouts, pullbacks, and weak equivalences are defined levelwise, 
$\mathbf{Fun} (\sD^{op},\sM)$ is proper when $\sM$ is. To see that 
$\mathbf{Fun} (\sD^{op},\sM)$ is a $\sV$-model category, it suffices to verify the pushout product characterization (\ref{tencoten2}) of  $\sV$-model categories, and this follows by adjunction from the fact that $\sM$ is a $\sV$-model category.  By definition, the functors $\text{ev}_d$
preserve fibrations and weak equivalences. For the last statement, we may as well replace
$\sD$ by $\sD^{op}$ and consider the model structure on $\mathbf{Fun} (\sD,\sM)$. 
By \myref{Gb}, its evaluation functor $\text{ev}_d$ has right adjoint $ G_d =F(\bY(d),-)$. 
The adjunction (\ref{biten}) implies that the functors $G_d$ preserve acyclic fibrations
when the functors ${\sD}(e,d)\odot (-)$ preserve cofibrations.  In turn, when that holds
the functors $\text{ev}_d$ preserve cofibrations.
\end{proof} 

\begin{rem}\mylabel{RepCof} By adjunction, since acyclic fibrations are level acyclic fibrations, 
if $i\colon M\rtarr N$ is a cofibration in $\sM$ then 
$F_di\colon F_d M\rtarr F_d N$ is a cofibration in $\mathbf{Fun} (\sD^{op},\sM)$ for any $d\in\sD$. 
Therefore, if $M$ is cofibrant, then each $F_dM$ is cofibrant. In particular, if 
$\sM=\sV$ and $\mathbf{I}$ is cofibrant, then each represented presheaf $\bY(d)$ 
is cofibrant in $\mathbf{Pre} (\sD,\sV)$. This need not hold in general, and that gives one
reason for preferring to enrich in monoidal categories $\sV$ with a cofibrant
unit object.  As discussed in \S\ref{How?} below, one might be able to use \myref{Muro} to arrange this.
\end{rem}

\begin{rem}\label{SmallObjArg}
By adjunction, the smallness condition on $F\mathcal{I}_{\sM}$ and $F\mathcal{J}_{\sM}$ means 
that the domains of maps $i\in \mathcal{I}_{\sM}$ and $j\in \mathcal{J}_{\sM}$ are small with 
respect to the level maps $A_d\rtarr X_d$ of a relative $F\mathcal{I}_{\sM}$ or $F\mathcal{J}_{\sM}$ cell complex $A\rtarr X$ in $\mathbf{Fun} (\sD^{op},\sM)$. This means that, in the arrow category of $\sM$, a map 
from a generating cofibration or acyclic cofibration into the levelwise colimit obtained from a relative cell complex factors through one of its terms. In practice, for example when $\sM=\sV$ is any of the usual compactly generated enriching categories, such as those listed in \myref{cofgen}, this condition holds trivially.  In topological situations, it often follows from the compactness of the domains of maps in 
$\mathcal{I}_{\sM}$ and $\mathcal{J}_{\sM}$.  In algebraic situations, the compactness is often even simpler since the relevant domains are free on a single generator.  We generally ignore the smallness condition, 
since it is not a serious issue in our context.
\end{rem}

\begin{rem}\mylabel{monoid}
The acyclicity condition has more substance, but it also usually holds 
in practice.  It clearly holds if $\sM$ satisfies the $\sD$-tensor axiom.
In particular, it holds if the functors ${\sD}(e,d)\odot (-)$ on $\sM$ 
are Quillen left adjoints or if all ${\sD}(d,e)$ are cofibrant in $\sV$. 
It holds for any $\sD$ if $\sM$ satisfies condition (ii) of the tensor axiom.

When the monoid axiom holds, \myref{level2} for $\sM=\sV$ is \cite[6.1]{SS2}; see also 
\cite[7.2]{SS2} for stable situations.  In topological situations, \myref{level2}
often applies even when the ${\sD}(d,e)$ are not cofibrant, the monoid axiom fails for 
$\sV$, and the functors ${\sD}(d,e)\odot(-)$ do not preserve level acyclic cofibrations.  
As noted earlier, axiomatizations of exactly what is needed to ensure this are given 
in \cite[4.5.8, 5.4.6]{MS}, which apply to all situations we have encountered. An
essential point is that in topology, and also in homological algebra, one has both 
classical cofibrations (HEP) and the cofibrations of the Quillen model structure,
and one can exploit the more general classical cofibrations to check the acyclicity
condition.
\end{rem} 

\begin{rem}\mylabel{Good}  One sometimes starts with a plain unenriched
category $\sC$ rather than an enriched category $\sD$.  To relate this to the enriched 
context, let 
$\sD =\mathbf{I}[\sC]$ be the $\sV$-category with the same object set as $\sC$ and 
with morphism objects $\mathbf{I}[\sC(d,e)]$. The composition is induced from that of $\sC$.  
If $\mathbf{I}$ is cofibrant in $\sV$, then $\mathbf{I}[S]$ is cofibrant for all sets $S$ and 
the acyclicity condition holds for $\mathbf{I}[\sC]$ and any $\sV$-model category $\sM$. Using that 
$\mathbf{I}[S]\odot V$ is the coproduct of copies of $V$ indexed by the elements of $S$, we see 
that the ordinary category $\mathbf{Fun} (\sC^{op},\sM)$ of unenriched presheaves in $\sM$ is 
isomorphic to the underlying category of the $\sV$-category $\mathbf{Fun} (\sD^{op},\sM)$.
\end{rem} 

In model category theory, diagram categories with discrete domain categories 
$\sC$ are often used to study homotopy limits and colimits \cite{DHKS, Hirsch, Hovey}. 
Shulman \cite{Shul} has given a study of enriched homotopy limits and colimits in 
$\sV$-model categories $\sM$, starting in the same general framework in which we are working.

\subsection{How and when to make the unit cofibrant?}\label{How?}

The unit $\mathbf{I}$ of $\sV$ may or may not be cofibrant in the model structure 
we start with on $\sV$, but the unit axiom of \myref{monoidal} holds in all cases of interest.   Often there is some cofibrant 
approximation $q\colon \bQ\mathbf{I} \rtarr \mathbf{I}$ such that  $\id \sma q\colon V\otimes \bQ\mathbf{I} \rtarr V\otimes \mathbf{I}\iso V$ is a  weak
equivalence for all objects $V\in \sV$, not just the cofibrant ones.   We then say that the very strong unit axiom holds.  As proven in \cite[Corollary 9]{Muro}, 
this condition holds for any cofibrant approximation $q$ if tensoring with a  cofibrant object preserves weak equivalences, which is often the case.  Thus
assuming the very strong unit axiom is scarcely more restrictive than assuming the unit axiom.  

This axiom is closely related to the theory of semicofibrant  objects developed by Lewis and Mandell \cite{LM}.  
Long after the posted draft of this paper appeared,  Muro wrote an illuminating paper \cite{Muro} in which he 
proved the following result. He does not assume that the monoidal structure on $\sV$ is symmetric, but we do assume that.

\begin{thm}[Muro]\mylabel{Muro}  Let $\sV$ be either a combinatorial or a cofibrantly generated monoidal model category satisfying the very strong unit axiom.
Then there is a combinatorial or cofibrantly generated monoidal model structure, denoted $\tilde{\sV}$, on the underlying category of $\sV$ which has the same weak
equivalences as $\sV$ and in which the unit object is cofibrant. The identify functor on $\sV$ is a left Quillen equivalence from
$\sV$ to $\tilde{\sV}$.  If $\sV$ satisfies the monoid axiom, then so does $\tilde V$.  If $\sV$ is left or right proper, then so is $\tilde{V}$.
\end{thm}

Thus $\tilde{\sV}$ has more cofibrations and therefore fewer fibrations than $\sV$. In the combinatorial case, the acyclic fibrations of $\sV$ are
the surjective acyclic fibrations of $\sV$.  In the cofibrantly generated case, the generating cofibrations are obtained by adding the morphism
$\emptyset \to \mathbf{I}$ to the generating cofibrations of $\sV$, and the generating acyclic cofibrations are obtained by adding a certain
well chosen acyclic cofibration between cofibrant approximations of $\mathbf{I}$ to the generating acyclic cofibrations of $\sV$.   

\begin{rem}\mylabel{unitquest}
This result raises some questions that we have not tried to answer.  First, under what conditions on $\sM$
is it true that a $\sV$-model structure  on $\sM$ is necessarily a $\tilde{\sV}$-model structure?   It is clear from
the definitions that a $\tilde{\sV}$-model structure is necessarily a $\sV$-model structure.  When $\sM$ is a
cofibrantly generated $\tilde{\sV}$-model category, \myref{level2} applies to show that presheaf categories with 
values in $\sM$ are also $\tilde{\sV}$-model categories.  It seems plausible that many of our model theoretic 
results that assume that $\mathbf{I}$ is cofibrant work without 
the unit assumption, by replacing $\sV$ by $\tilde{\sV}$ in their proofs.
\end{rem}

The cofibrancy of unit condition is related to the categorical fact that the automorphism group of 
the unit object of a symmetric monoidal category is commutative.   In stable homotopy theory at
least, that in effect forces interest in examples where the unit is not cofibrant.

\begin{rem}\mylabel{Vpos}
In the category $\sV$ of $S$-modules \cite{EKMM}, the unit is not cofibrant and every object is fibrant. 
For reasons explained in \cite[Remark 11.2]{Rant1}, we cannot hope to make the unit cofibrant 
and keep the property that every object is fibrant; compare  \cite[Example 6]{Muro}. In the categories $\sV$
of symmetric and orthogonal spectra, the unit is cofibrant  and, to deal with commutative monoids,
it is essential to change the given model category to a ``positive'' model structure $\sV_+$ in which the unit is not
cofibrant.  Such situations are alluded to in \myref{subtle}.  Applying 
\myref{Muro} to $\sV_+$ results in a model structure that interpolates between $\sV$ and $\sV_+$: there are left 
Quillen equivalences $\sV_+ \rtarr \tilde{\sV}_+ \rtarr \sV$; compare \cite[Examples 2, 5]{Muro}. 
\end{rem}

\myref{Muro} does not change the need for the weakly unital $\sV$-categories of \S3.5 since
those address categorical rather than just model theoretic issues.

\section{Appendix: Enriched presheaf categories}\label{diagram}

\subsection{Categories of enriched presheaves}\label{CatPre}

We record some categorical observations about categories $\mathbf{Fun} (\sD^{op},\sM)$ of
enriched presheaves.  Remember that $\mathbf{Pre} (\sD,\sV) = \mathbf{Fun} (\sD^{op},\sV)$.
We ignore model structures
in this section, so we only assume that $\sV$ is a bicomplete closed symmetric
monoidal category and $\sM$ is a bicomplete $\sV$-category.  Then $\mathbf{Fun} (\sD^{op},\sM)$
is a $\sV$-category. The enriched hom $\ul{\mathbf{Fun} }(\sD^{op},\sM)(X,Y)$ is the 
equalizer in $\sV$  displayed in the diagram
\begin{equation}\label{PXY}
\xymatrix@1{
 \ul{\mathbf{Fun} }(\sD^{op},\sM)(X,Y) \ar[r] & \prod_{d} \ul{\sM}(X_d,Y_d) \ar@<1ex>[r] \ar@<-1ex>[r] 
 & \prod_{d,e}\ul{\sM}({\sD}(e,d) \odot X_d,Y_e).\\} 
\end{equation} 
The parallel arrows are defined using the evaluation maps
\[ {\sD}(e,d)\odot X_d\rtarr X_e\ \ \text{and}\ \  
{\sD}(e,d)\odot Y_d\rtarr Y_e \]
of the $\sV$-functors $X$ and $Y$, in the latter case after composition with
\[ {\sD}(e,d)\odot (-)\colon \ul{\sM}(X_d,Y_d)\rtarr 
\ul{\sM}({\sD}(e,d)\odot X_d,{\sD}(e,d)\odot Y_d).\]
The $\sV$-category $\mathbf{Fun} (\sD^{op},\sM)$ is bicomplete, with colimits, limits,
tensors and cotensors defined in the evident objectwise fashion; in particular,
\[ (X\odot V)_d = X_d\odot V\ \ \text{and} \ \ F(V,X)_d = F(V,X_d).\]

For clarity below, the reader should notice the evident identifications
\[ \mathbf{Fun} (\sD^{op},\sM^{op}) \iso \mathbf{Fun} (\sD,\sM)^{op} \ \ \text{and hence}\ \ 
\mathbf{Fun} (\sD,\sM^{op}) \iso \mathbf{Fun} (\sD^{op},\sM)^{op}. \] 
Applied levelwise, the functors $(-)\odot M\colon\sV\rtarr \sM$ and
$F(-,M)\colon \sV^{\text{op}}\rtarr \sM$ for varying $M$ 
induce $\sV$-functors
\[\odot \colon \mathbf{Fun} (\sD^{op},\sV) \otimes \sM\rtarr \mathbf{Fun} (\sD^{op},\sM) \]
and 
\[ F\colon \mathbf{Fun} (\sD^{op},\sV) \otimes \sM\rtarr \mathbf{Fun} (\sD,M).\]
Similarly, the functors $\ul{\sM}(M,-)$ and $\ul{\sM}(-,M)$ induce 
$\sV$-functors, denoted $\ul{\sM}(-,-)$,
\[\sM^{op}\otimes F(\sD^{op},\sM) \rtarr \mathbf{Fun} (\sD^{op},\sV)
\ \ \text{and}\ \  \mathbf{Fun} (\sD^{op},\sM) \otimes \sM\rtarr 
\mathbf{Fun} (\sD,\sV). \]

Now let
\[X\in \mathbf{Fun} (\sD^{op},\sV), \ \ Y\in \mathbf{Fun} (\sD^{op},\sM), \ \ \text{and}
Z\in \mathbf{Fun} (\sD,M).\]  The categorical tensor product (specializing
left Kan extension) of the
contravariant functor $X$ and the covariant functor $Z$ on $\sD$ gives the
object $X\odot_{\sD} Z\in \sM$ displayed in the coequalizer diagram
\begin{equation}\label{FWXO}
\xymatrix@1{
\coprod_{d,e}X_e\otimes {\sD}(d,e)\odot Z_d \ar@<1ex>[r] \ar@<-1ex>[r] &
\coprod_d X_d\odot Z_d\ar[r] & X\odot_{\sD} Z.\\}
\end{equation}
The parallel arrows are defined using the evaluation maps of $X$ and $Z$
and the isomorphism (\ref{trans}). Similarly, the categorical
hom of the contravariant functors $X$ and $Y$ gives the object
$F_{\sD}(X,Y)\in\sM$ displayed in the analogous equalizer diagram
\begin{equation}\label{FWX}
\xymatrix@1{
 F_{\sD}(X,Y) \ar[r] & \prod_{d} F(X_d,Y_d) \ar@<1ex>[r] \ar@<-1ex>[r] 
 & \prod_{d,e} F({\sD}(e,d) \otimes X_d,Y_e).\\} 
\end{equation}
With these constructions, we have adjunctions analogous to those of (\ref{biten}):
\begin{equation}\label{WMXO}
\ul{\sM}(X\odot_{\sD} Z,M) \iso \ul{\mathbf{Fun} }(\sD^{op},\sV)(X,\ul{\sM}(Z,M))
\iso \ul{\mathbf{Fun} }(\sD,\sM)(Z,F(X,M))
\end{equation}  
and
\begin{equation}\label{WMX} 
\ul{\mathbf{Fun} }(\sD^{op},\sM)(X\odot M, Y) \iso \ul{\mathbf{Fun} }(\sD^{op},\sV)(X,\ul{\sM}(M,Y)) 
\iso \ul{\sM}(M,F_{\sD}(X,Y)).
\end{equation}
Applying $\sV(\mathbf{I},-)$, there result ordinary adjunctions, with hom sets
replacing hom objects in $\sV$, that are 
analogous to those displayed in (\ref{biten2}).

The proofs of the adjunctions in \myref{Fb} and \myref{Gb} are now immediate.
By (\ref{WMX}) and the enriched Yoneda lemma, we have
\begin{eqnarray*} 
\ul{\mathbf{Fun} }(\sD^{op},\sM)(F_d M, Y) &\iso & \ul{\mathbf{Fun} }(\sD^{op},\sV)(\bY(d),\ul{\sM}(M,Y)) \\
 &\iso & \ul{\sM}(M,Y_d)= \ul{\sM}(M,\text{ev}_d(Y)).
\end{eqnarray*}
Dually, by (\ref{WMXO}) and the enriched Yoneda lemma, we have
\begin{eqnarray*} 
\ul{\mathbf{Fun} }(\sD,\sM)(Z, G_dM ) & \iso & \ul{\mathbf{Fun} }(\sD^{op},\sV) (\bY(d),\ul{\sM}(Z,M)) \\
  &\iso & \ul{\sM}(Z_d,M) = \ul{\sM}(\text{ev}_d(Z),M).  
\end{eqnarray*}

Since limits, colimits, tensors, and cotensors in $\mathbf{Fun} (\sD^{op},\sM)$ are defined levelwise, the functors ${\text{ev}_d}$ preserve all of these, and so do their adjoints $F_d$ and $G_d$. 

\subsection{Constructing $\sV$-categories over a full $\sV$-subcategory of $\sV$}\label{overV}

In the sequel \cite{GM2} we are especially interested in finding calculationally 
accessible domain $\sV$-categories $\sC$ for categories of presheaves $\sV^{\sC}$ that are 
equivalent to categories of presheaves $\mathbf{Pre} (\sD,\sV)$, where $\sD$ is a well chosen full 
$\sV$-subcategory of an ambient $\sV$-category $\sM$.  Of course, for that purpose we are not at all concerned with the underlying categories of $\sC$ and $\sD$.  In \S\ref{overM}, 
we shall give a theoretical description of all such $\sV$-maps $\sC\rtarr \sD$, where 
$\sD$ is preassigned.

Here we restrict attention to $\sM = \sV$ and give a simple general way of 
constructing a $\sV$-map $\ga\colon \sC\rtarr \sD$ where $\sC$ is a small
$\sV$-category and $\sD$ is a full $\sV$-category of $\sV$ whose objects 
are specified in terms of $\sC$.  Despite its simplicity, this example will play
a key role in the sequel \cite{GM2}.  

\begin{con} Fix an object $e\in \sC$.  In the applications,
$e$ is a distinguished object with favorable properties. Let $\sD$ be the
full $\sV$-subcategory of $\sV$ whose objects are the $\sC(e,c)$ for $c\in \sC$. 
We define a $\sV$-functor $\ga\colon \sC\rtarr \sD$ such that $\ga(c)= \sC(e,c)$ 
on objects. The map
\[ \ga\colon \sC(b,c) \rtarr \sD(b,c) = \ul{\sV}(\sC(e,b),\sC(e,c)) \]
in $\sV$ is the adjoint of the composition 
\[ \circ\colon \sC(b,c)\otimes \sC(e,b)\rtarr \sC(e,c). \]
The diagrams 
\[ \xymatrix{
& \mathbf{I} \ar[dl]_-{\et} \ar[dr]^{\et}  & \\
\sC(c,c) \ar[rr]_-{\ga} & & \ul{\sV}(\sC(e,c),\sC(e,c))\\} \]
and
\[ \xymatrix{
\sC(b,c)\otimes \sC(a,b) \ar[r]^-{\ga\otimes \ga} \ar[d]_{\circ}& 
\ul{\sV}(\sC(e,b),\sC(e,c))\otimes \ul{\sV}(\sC(e,a),\sC(e,b)) \ar[d]^{\circ}\\
\sC(a,c)\ar[r]_-{\ga} & \ul{\sV}(\sC(e,a),\sC(e,c))\\} \]
commute since their adjoints
\[ \xymatrix{
& \mathbf{I}\otimes \sC(e,c) \ar[dl]_{\et\otimes \id} \ar[dr]^{\iso} &\\
\sC(c,c)\otimes \sC(e,c) \ar[rr]_-{\circ} & & \sC(e,c)\\} \]
and 
\[ \xymatrix{
\sC(b,c)\otimes \sC(a,b)\otimes \sC(e,a) \ar[r]^-{\id\otimes \circ} 
\ar[d]_{\circ\otimes\id} & \sC(b,c)\otimes \sC(e,b) \ar[d]^{\circ}\\
\sC(a,b)\otimes \sC(e,a) \ar[r]_{\circ} & \sC(a,c)\\}  \]
commute. To see that the last diagram is adjoint to the second, observe that
$$\ga\otimes\ga = (\ga\otimes \id)\circ(\id\otimes \ga).$$
\end{con}

\begin{rem}\mylabel{gammalax} When $\sC$ is a symmetric monoidal $\sV$-category with unit 
object $e$ and product $\Box$, the composite $\ga\colon \sC\rtarr \sD\subset \sV$ is a
lax symmetric monoidal $\sV$-functor.  The data showing this are the unit map
$\mathbf{I}\rtarr \sC(e,e) = \ga(e)$ and the product map
\[ \Box \colon \ga(b)\otimes \ga(c) = \sC(e,b)\otimes \sC(e,c) 
\rtarr \sC(e,b\Box c) = \ga(b\Box c), \]
where we have used the canonical isomorphism $e\Box e\iso e$. 
\end{rem}

\subsection{Characterizing $\sV$-categories over a full $\sV$-subcategory of $\sM$}\label{overM}
We use the first adjunction of (\ref{biten}) to characterize the $\sV$-categories 
$\ga\colon \sC\rtarr \sD$ over any full $\sV$-subcategory $\sD$ of a $\sV$-category $\sM$. 
Technically, we do not assume that $\sM$ is bicomplete, but we do assume the adjunction, 
so that we have tensors; we write them as $V\odot M$.  Let $\sV$-$\sC\!at/ \sD$ be the category whose objects are the $\sV$-functors 
$\ga\colon \sC\rtarr \sD$ that are the identity on objects and whose morphisms are the 
$\sV$-functors $\al\colon \sC\rtarr \sC'$ such that $\ga'\com \al = \ga$.

Consider the following data.
\begin{enumerate}[(i)]
\item For each pair $(d,e)$ of objects of $\sD$, an object $\sC(d,e)$ of $\sV$ and an
``evaluation map''  $\epz\colon \sC(d,e)\odot d \rtarr e$ in $\sM$ with adjoint map (in $\sV$)
\[ \ga\colon \sC(d,e)\rtarr \sD(d,e) = \ul{\sM}(d,e). \]
We require the following associativity diagram to commute for $(b,c,d,e)\in \sD$.
 \[ \xymatrix{
(\sC(d,e)\otimes \sC(c,d) \otimes \sC(b,c))\odot b \ar[d]_{(\mu\otimes\mbox{id})\odot\mbox{id}} 
\ar[rr]^-{(\mbox{id}\otimes \mu)\odot \mbox{id}}
& & (\sC(d,e)\otimes(\sC(b,d)) \odot b \ar[d]^{\iso}\\
(\sC(c,e)\otimes \sC(b,c)) \odot b \ar[d]_{\iso}
& &\sC(d,e)\odot(\sC(b,d)\odot b) \ar[d]^{\mbox{id}\odot \epz}\\
\sC(c,e) \odot (\sC(b,c)\odot b) \ar[d]_{\mbox{id}\odot \epz} 
& &\sC(d,e)\odot d \ar[d]^{\epz}\\
\sC(c,e) \odot c \ar[rr]_-{\epz}
& & e\\} \]
Diagram chasing then shows that, under the canonical isomorphism of their sources, the composites 
in the diagram agree with the following composite of evaluation maps.
\[ \sC(d,e)\odot \sC(c,d) \odot \sC(b,c)\odot b \rtarr
\sC(d,e)\odot\sC(c,d)\odot c \rtarr
\sC(d,e)\odot d \rtarr e.  \]
\item For each object $d$ of $\sD$ a ``unit map'' $\et\colon \mathbf{I}\rtarr \sC(d,d)$ in $\sV$
such that the following diagram commutes.
\[ \xymatrix{
& \mathbf{I}\odot d\ar[dl]_{\et\odot\text{id}} \ar[dr]^{\iso}  & \\
\sC(d,d)\odot d \ar[rr]_{\epz} & & d\\} \]
Using the isomorphism (\ref{trans}), which depends on having the enriched adjunction (\ref{biten}) 
in $\sM$, we define composition maps
\[ \mu\colon \sC(d,e)\otimes \sC(c,d) \rtarr \sC(c,e) \] 
in $\sV$ as the adjoints of the following composites of evaluation maps in $\sM$.
\[ \xymatrix@1{
({\sC}(d,e)\otimes {\sC}(c,d))\odot c \iso {\sC}(d,e)\odot ({\sC}(c,d)\odot c)
\ar[r]^-{\mbox{id}\odot\epz} & {\sC}(d,e)\odot d \ar[r]^-{\epz} & e.\\} \] 
\end{enumerate}

\begin{prop} There is an isomorphism between $\sV$-$\sC\!at/ \sD$ and the 
category whose objects consist of the data specified in (i) and (ii) above and whose
morphisms $\al\colon \sC\rtarr \sC'$ are given by maps $\al\colon \sC(d,e) \rtarr \sC'(d,e)$
such that the following diagrams commute (in $\sM$ and $\sV$ respectively).
\[ \xymatrix{
\sC(d,e) \odot d \ar[dr]_{\epz}\ar[rr]^-{\al\odot\text{id}} & & \sC'(c,d)\odot e \ar[dl]^{\epz'}\\
& e &\\} 
\ \ \ \mbox{and} \ \ \ 
\xymatrix{
& \mathbf{I} \ar[dl]_{\et}  \ar[dr]^{\et'}& \\
\sC(d,d) \ar[rr]_-{\al} & & \sC'(d,d)\\} \]
\end{prop}
\begin{proof}  For an object $\sC$ of the category defined in the statement we easily verify from the given data that $\sC$
is a $\sV$-category with the specified unit and composition maps and that the maps $\ga$ 
together with the identity function on objects specify a $\sV$-functor $\sC\rtarr \sD$. 
Conversely, for a $\sV$-functor $\ga\colon \sC\rtarr \sD$ that is the identity on objects,
we obtain data as in (i) and (ii) by use of the adjunction (\ref{biten}). This correspondence
between objects carries over to a correspondence between morphisms.
\end{proof}

\subsection{Remarks on multiplicative structures}\label{mult}
Our results in this paper, like nearly all of the results in the literature on 
replacing given model categories by equivalent presheaf categories, ignores any given multiplicative structure on $\sM$.  The following 
observations give a starting point for a study of products, but we shall not 
pursue this further here. There are several problems.  For starters, the hypotheses in the
following remark are natural categorically, but they are seldom satisfied
in the applications. Moreover, the assumption here that $\de$ is op-lax clashes with
the conclusion that $\ga$ is lax in \myref{gammalax}. In practice, it cannot be
expected that either is strong symmetric monoidal.

\begin{rem} Suppose that $\sD$ is symmetric $\sV$-monoidal with product $\oplus$ and 
unit object $e$.  Then $\mathbf{Pre} (\sD,\sV)$ is symmetric $\sV$-monoidal with product $\otimes$
and unit object $\bY(e)$.  For $X,Y\in \mathbf{Pre} (\sD,\sV)$, we have the evident external product 
$\bar{\otimes}\colon \sD\otimes \sD\rtarr \sV$, which is given on objects by
$(X\bar{\otimes} Y)(b,c) = X(b)\otimes Y(c)$, and left Kan extension along $\oplus$
gives the product $X\otimes Y\in \mathbf{Pre} (\sD,\sV)$. It is characterized by the $\sV$-adjunction 
\[ \ul{\mathbf{Pre} }(\sD,\sV)(X\otimes Y, Z) \iso \ul{\mathbf{Pre} }(\sD\otimes \sD,\sV)(X\bar{\otimes} Y, Z\circ \oplus). \]
\end{rem}

\begin{rem} Now suppose further that $\sM$ is symmetric $\sV$-monoidal with product $\Box$ and unit
object $\mathbf{J}$ and that the $\sV$-functor $\de\colon \sD\rtarr \sM$ is op-lax 
symmetric $\sV$-monoidal, so that we are given a map $\ze \colon \de e\rtarr \mathbf{J}$ 
in $\sM$ and a natural $\sV$-map $$\psi\colon \de(c\oplus d) \rtarr \de c \Box \de d.$$  
Then the functor $\bU\colon \sM\rtarr \mathbf{Pre} (\sD,\sV)$ is lax symmetric monoidal and therefore 
its left adjoint $\bT$ is op-lax symmetric monoidal. The data showing this are a unit map
$\et\colon \bY(e) \rtarr \bU \mathbf{J}$ in $\mathbf{Pre} (\sD,\sV)$ and a natural $\sV$-map 
$\ph\colon \bU M\otimes \bU N \rtarr \bU(M\Box N)$.  Recall that 
$(\bU M)(d) =\ul{\sM}(\de d,M)$. The map $\et$ is given by the composite maps
\[ \bY(e)(d) = \sD(d,e) \rtarr \ul{\sM}(\de d,\de e) \rtarr \ul{\sM}(\de d,\mathbf{J}) \]
induced by $\de$ and $\ze$.  The natural $\sV$-map $\ph$ is adjoint to the natural 
$\sV$-map
\[  \bU M \bar{\otimes} \bU N \rtarr \bU (M \Box N)\circ \oplus \]
given by the composite maps
\[ \ul\sM(\de c,M) \otimes \ul\sM(\de d,N) 
\rtarr \ul\sM(\de c\Box \de d, M \Box N)
\rtarr \ul\sM(\de (c\oplus d), M\Box N) \]
induced by $\Box$ and $\ph$.
\end{rem}

\end{document}